\documentclass[10pt]{amsart}
\usepackage{appendix}
\usepackage[colorlinks]{hyperref}
\usepackage{subfigure}
\usepackage{mathtools}
\usepackage{srcltx}

\newcommand{\R}        {\mathbb {R}}

\newcommand{\grad}     {\nabla}

\newcommand{\del}      {\partial}

\renewcommand{\vec}[1]{\boldsymbol{#1}}
\DeclareMathOperator{\Div}{div}

\newtheorem{definition}{Definition}[section]
\newtheorem{theorem}{Theorem}[section]
\newtheorem{proposition}{Proposition}[section]
\newtheorem*{thma}{Theorem A}
\newtheorem{lemma}{Lemma}[section]
\newtheorem{corollary}{Corollary}[section]
\newtheorem{remark}{Remark}[section]
\newcommand{\rlemma}[1]{Lemma~\ref{#1}}
\newcommand{\rth}[1]{Theorem~\ref{#1}}

\newcommand{\rrem}[1]{Remark~\ref{#1}}
\newcommand{\rprop}[1]{Proposition~\ref{#1}}

\theoremstyle{definition} 

\newtheorem*{maintheorem*}{Main Theorem}

\allowdisplaybreaks

\numberwithin{equation}{section}
\numberwithin{figure}{section}
\numberwithin{table}{section}

\newcounter{asnr}

{\ifnum\value{asnr}=0 \stepcounter{asnr} 
  \begin{enumerate}[label=\textbf{A}.\arabic{enumi}]
    \else
    \begin{enumerate}[label=\textbf{A}.\arabic{enumi},resume] \fi}
{\end{enumerate}}

\title[]{ Asymptotics of eigenstates of elliptic problems  with mixed boundary data on domains tending to infinity}

\author[M. Chipot]{M.\ Chipot} \address[Michel Chipot]{\newline Institut f\"ur Mathematik,\ Universit\"at Zurich, \newline
Winterthurerstr. 190,\ CH-8057 Z\"urich,
Switzerland.}
\email[]{m.m.chipot@math.uzh.ch}

\author[P. Roy]{P.\ Roy} \address[Prosenjit Roy]{\newline Institut f\"ur Mathematik,\ Universit\"at Zurich, \newline
Winterthurerstr. 190,\ CH-8057 Z\"urich,
Switzerland.}
\email[]{prosenjit.roy@math.uzh.ch}

\author[I. Shafrir]{I.\ Shafrir} \address[Itai Shafrir ]{\newline
Department of Mathematics, Technion - Israel Institute of Technology
\newline 32000, Haifa, Israel}
\email[]{shafrir@math.technion.ac.il}

\date{\today}

\begin{document}
\maketitle
% \centerline{\bf ASYMPTOTIC BEHAVIOUR OF THE EIGENSTATES OF ELLIPTIC}
% \smallskip
% \centerline{\bf PROBLEMS WITH MIXED BOUNDARY CONDITIONS }

% \vskip 1 cm

% \centerline{M. CHIPOT, P. ROY AND I. SHAFRIR}

% \vskip 1 cm
ABSTRACT. We analyze the asymptotic behavior of eigenvalues and
eigenfunctions of an elliptic operator with mixed boundary conditions
on cylindrical domains when the length of the cylinder goes to
infinity. 
We  identify the correct limiting problem and show in particular, that
in general the limiting behavior is very different from the one for
the Dirichlet boundary conditions.  
\vskip 1 cm

%\maketitle

\section{Introduction}\label{intro}
Let $\omega$ be a bounded open set in $\R^{n-1}$. For every
$\ell>0$ set $\Omega_\ell = (-\ell, \ell)\times{}\omega$ and
write each $x\in\Omega_\ell$ as $x=(x_1,X_2)$ with $X_2=(x_2,\ldots,x_n)$. We assume that
the matrices 
$$A(X_2)=\begin{pmatrix}a_{11}(X_2)  & A_{12}(X_2)\\
A_{12}^t(X_2) & A_{22}(X_2)
  \end{pmatrix}
$$
are uniformly elliptic and uniformly bounded
on $\omega$ (precise assumptions will be made in Section~\ref{preliminaries}). 
 The limiting behavior, when $\ell$ goes to infinity, of the eigenvalues and
 eigenfunctions of the elliptic operator $-\Div(A(X_2)\nabla u)$ on
 $\Omega_\ell$ with
 zero Dirichlet boundary conditions,  was studied by Chipot and
 Rougirel in \cite{dir}. We shall recall below one of their main
 results that was the principal motivation for the current paper.
Let $\mu{}^k$ and 
$\sigma_\ell^k$ denote, respectively,  the $k$th eigenvalues for the problems  
\begin{equation}
\label{eq:3}
\left\{
\begin{aligned}
  -&\Div(A_{22}(X_2)\nabla u)=\mu{}u \quad \text{ in }\omega,\\
     &u=0 \quad \text{ on }\partial\omega,
\end{aligned}
\right.
\end{equation}
and 
\begin{equation}
\label{eq:4}
\left\{
\begin{aligned}
  -&\Div(A(X_2)\nabla u)=\sigma u \quad\text{ in }\Omega_\ell,\\
     &u=0 \quad\text{ on }\partial\Omega_\ell.
\end{aligned}
\right.
\end{equation}
 The following relation between problem \eqref{eq:4} (for large $\ell$) and
 problem \eqref{eq:3} was established in  \cite{dir}.
\begin{thma}[Chipot-Rougirel]\label{dirichh}
\begin{equation}\label{upplow}
\mu^1 \leq \sigma_\ell^1 \leq \mu^{1} + \frac{C}{\ell^2}\,,
\end{equation}
where $C$ is a constant independent of $\ell$.
\end{thma}

 The main goal of the present article is to study the analogous
 problem for {\em mixed} boundary conditions, at least for $k=1$.
 Let us write $\partial\Omega_\ell=\Gamma_\ell\cup \gamma_\ell$ where
 \begin{equation}\label{decom}
\Gamma_\ell = \{-\ell, \ell\} \times{}\omega  \text{ and }   \gamma_\ell = (-\ell, \ell)\times{}\partial\omega,
\end{equation}
 and denote by $\lambda_\ell^k$   the $k$th eigenvalue for the mixed
 Neumann-Dirichlet problem
\begin{equation}
\label{eq:5}
\left\{
\begin{aligned}
  -&\Div(A(X_2)\nabla u)=\sigma u \quad\text{ in }\Omega_\ell,\\
     &u=0 \quad\text{ on }\gamma_\ell,\\
     & (A(X_2)\nabla u).\nu=0 \quad\text{ on }\Gamma_\ell.
\end{aligned}
\right.
\end{equation}
 One of our main results establishes that $\lim_{\ell\to\infty}\lambda_\ell^1$ exists, but in
  general it is strictly smaller than $\mu^1$. This ``gap phenomenon'' is
  explained by the appearance of boundary effects near
  $\Gamma_\ell$. To gain better understanding of these effects we are
  led to consider first the limit $\lim_{\ell\to0}\lambda_\ell^1$. 
Asymptotic behavior of elliptic problems set on
domains shrinking  to zero in some directions are generally known as ``Dimension Reduction" problems and are addressed in \cite{acer,ciar,fon} and in a setting particularly
suitable for us, in \cite{braides} . Our work establishes a somewhat
surprising connection  between the theory of dimension reduction
(i.e., ``$\ell\to0$'') and the
theory for ``$\ell \to\infty$''. 
\smallskip

In order to have a more precise description of the boundary effects and
to characterize the value of the limit
$\lim_{\ell\to\infty}\lambda_\ell^1$, we introduce  eigenvalue
problems on the two semi-infinite cylinders
$\Omega_\infty^+=(0,\infty)\times\partial\omega$ and $\Omega_\infty^-=(-\infty,0)\times\partial\omega$, with mixed boundary
conditions. Let
$\nu_\infty^\pm$ denote the first eigenvalue for the operator 
$-\Div(A(X_2)\nabla u)$ on $\Omega_\infty^\pm$ with zero boundary
condition on  the lateral part of the boundary
$\partial\Omega_\infty^\pm$. One might be tempted to expect that the
equality $\nu_\infty^+=\nu_\infty^-$ always hold because of
``symmetry considerations''. However, as we shall see in
Section~\ref{sec:cylinder}, this equality is false in general. Our main results are summarized in the
next theorem, that combines the results of \rth{th1} and
\rth{th:lim-lam}. We denote by $W_1$ the positive normalized
eigenfunction corresponding to $\mu^1$.
\begin{maintheorem*}
  We have $\lim_{\ell\to\infty}\lambda_\ell^1=\min(\nu_\infty^+,\nu_\infty^-)\,.$
If $A_{12}.\nabla W_1\not\equiv  0\text{ a.e. on } \omega,$ then
$\lim_{\ell\to\infty}\lambda_\ell^1<\mu^1$. Otherwise,
$\lambda_\ell^1=\mu^1$, $\forall\ell$.
\end{maintheorem*}
\smallskip

Many  problems of the type ``$\ell\to\infty$'' were
studied in the past. Besides the eigenvalue problem already
mentioned~\cite{dir}, these  include elliptic and parabolic equations,
variational inequalities and systems, see \cite{b,c,d,e,f,g,z}. In
all these problems it is found that
the limit is characterized by the solution of the
corresponding problem on the section $\omega$. We emphasize that the
limiting behavior in our problem is very different.
\smallskip 

The paper is organized as follows. In Section~\ref{preliminaries} we give the
main definitions and notation needed in the subsequent sections. In
Section \ref{model}  we illustrate the gap phenomenon in a simple
model case
where $\omega=(-1,1)$ and $A$ is a $2\times 2$ matrix with constant
coefficients, namely,
$A = A_\delta =
\begin{pmatrix}
1 & \delta \\
\delta & 1
\end{pmatrix}$. In Section~\ref{ge} we prove the gap phenomenon for
the general case. In Section~\ref{sec:lim} we prove that the limit
$\lim_{\ell\to\infty}\lambda_\ell^1$ exists, and identify its value
using the eigenvalue problems on the semi-infinite cylinders
$\Omega_\infty^+$ and $\Omega_\infty^-$. In
Section~\ref{sec:cylinder} we investigate further the problem on a
semi-infinite cylinder and use it to give a more precise description
of the first eigenfunction $u_\ell$ for large $\ell$. In the last
section, Section~\ref{sec:further},  we address briefly two natural
related problems. First, we present a result on the asymptotics of the
second eigenvalue $\lambda_\ell^2$ as $\ell$ goes to infinity (under some symmetry assumption on the matrix $A$). Second,
we give a partial result for the more general case of a domain
becoming large in several directions.

% The work of this paper is organized as follows. \ In the next section we introduce most of the notation that will be used in this paper. 
% \ In  Section \ref{pre} we present some preliminaries about the problems  \eqref{P} and 
% \eqref{cross}.\ In Section \ref{model},\  we provide  results about the first eigenvalue 
% $\lambda_\ell^1 $ as $\ell \to \infty$, \ 
% but for a model problem,\ to explain the ideas used more clearly and explicitly.\ In Section \ref{ge} we prove
% similar results, in a general setting.\ Section \ref{Issue} is devoted to the study of the structure 
% of the first eigenfunction $u_\ell$. \ In Section \ref{label2} we look at the asymptotic of the second
% eigenvalue $\lambda_\ell^2$ of the problem \eqref{P}.\ We will see that the study of the asymptotic of $\lambda_\ell^2$ is based
%  on  a prior study of a different  problem,\ see \eqref{S}.\ 

\section{Preliminaries}\label{preliminaries}
For each $\ell>0$ consider  $\Omega_\ell = (-\ell, \ell)  \times{}\omega$ with $\omega$ a
bounded domain in $\R^{n-1}$ as in the Introduction. The lateral part
of $\partial\Omega_\ell$ and the remaining part of the cylinder (i.e., the two
ends) will be denoted by $\gamma_\ell$ and $\Gamma_\ell$, respectively. 
Let us denote by $H^1(\Omega_\ell)$ and $H_0^1(\Omega_\ell)$ the usual spaces of functions  defined by 
\begin{equation*}
H^1(\Omega_\ell) = \left\{ v \in L^2(\Omega_\ell) |\ \partial_{x_i}v \in L^2(\Omega_\ell),  i = 1, 2, \ldots  ,n \right\},
\end{equation*}
and
\begin{equation*}
H_0^1(\Omega_\ell) = \left\{v \in H^1(\Omega_\ell) | \ v = 0 \text{ on
  } \  \del\Omega_\ell \right\},
\end{equation*}
or in a more precise way, $H_0^1(\Omega_\ell)$ is the closure of
$C^\infty_c(\Omega_\ell)$ in $H^1(\Omega_\ell)$.
The space  $H_0^1(\Omega_\ell)$ is equipped with the norm
\begin{equation}
\label{eq:6}
\| \nabla v \|^2_{2, \Omega_\ell} = \int_{\Omega_\ell} |\nabla v|^2.
\end{equation}
A suitable space  for our problem is
\begin{equation*}
V(\Omega_\ell) = \left\{ v \in H^1(\Omega_\ell) \ | \ v = 0  \text{ on
  }\gamma_\ell  \right\},
\end{equation*}
 where the boundary condition should be interpreted in the sense of traces.
Thanks to the Poincar\'e inequality, $V(\Omega_\ell)$ becomes an Hilbert space
when equipped with the  norm \eqref{eq:6}.  For later use we define
the sets 
\begin{equation}
\label{eq:36}
\Omega_\ell^+ = [0, \ell)\times{}\omega \text{ and }  \Omega_\ell^-= (-\ell,0)\times{}\omega,
\end{equation}
We decompose $\Gamma_\ell$ (see \eqref{decom}) into two parts as
$\Gamma_\ell = \Gamma_\ell^+ \cup \Gamma_\ell^-$, where 
\begin{equation}\label{dec}
\Gamma_\ell^+ = \{\ell \}\times\omega ~\text{ and }~   \Gamma_\ell^- =\{ -\ell \}\times\omega\,.
\end{equation}
 Similarly, for the lateral part of $\partial\Omega_\ell$ we define,
 \begin{equation}
   \label{eq:40}
   \gamma_\ell^+ = (0,\ell)\times\partial\omega ~\text{ and }~   \gamma_\ell^- =(-\ell,0)\times\partial\omega\,.
 \end{equation}
 We shall be concerned with the operator
 $-\Div(A(X_2)\nabla u)$ where, for each $X_2\in \omega$,
\begin{equation*}\label{matrix}
\begin{aligned}
  A(X_2)=
\begin{pmatrix}
    a_{11}(X_2)  & A_{12}(X_2) \\
A_{12}^t(X_2) & A_{22}(X_2)
  \end{pmatrix}\end{aligned}
 \end{equation*}
 is a symmetric  $n \times n $ matrix,  $a_{11}\in\R$,
$A_{12}$ is  a $1 \times(n-1)$ matrix and $A_{22}$ is a  $(n-1)\times
(n-1)$ matrix. The components of $A(X_2)$ are assumed to be bounded measurable functions on $\omega$ and we assume the following bound 
\begin{equation}\label{mod}
\| A(X_2)\| \leq C_A\,,~~a.e. \,X_2 \in \omega,
\end{equation}
for the Euclidean operator norm. 
We also assume that $A(X_2)$ is uniformly elliptic and
 denote by $\lambda_A$ the largest positive number for which the
following inequality holds,
\begin{equation}\label{e}
A(X_2)\xi . \xi \geq \lambda_A |\xi|^2\,,~~\forall \xi \in \R^n, \ a.e. \,X_2 \in \omega.
\end{equation}
\bigskip

% We would like to consider the eigenvalue problem
% \begin{equation}\label{eigen1}
% \left. \begin{array}{llll}
% u_\ell \in V(\Omega_\ell), \\
% \int_{\Omega_\ell}A \nabla u_\ell . \nabla v dx = \lambda_\ell\int_{\Omega_\ell}u_\ell v dx \hspace{4mm} \forall v\in V(\Omega_\ell).
% \end{array}\right\}
% \end{equation}
% To be more precise,\ we say that $\lambda_\ell$ is an eigenvalue to \eqref{eigen1} if there exists a 
% function $u_\ell \neq 0$ solution to the above system.
 The weak formulation of the eigenvalue problem \eqref{eq:3} is to
 find $u\in  H_0^1(\omega)\setminus\{0\}$ and  $\mu{}\in\R$ such that 
\begin{equation}\label{eigen2}
\int_{\omega}(A_{22} \nabla u). \nabla v \,dX_2 = \mu\int_{\omega}uv\,dX_2\,,~~\forall v\in H_0^1(\omega)\,.
\end{equation}
Denote by  $\mu^1$ the first eigenvalue of the problem \eqref{eigen2}
with  
the corresponding  normalized eigenfunction  $W_1$, i.e.,
$\int_\omega |W_1|^2=1$. 
It is well known that $\mu{}^1$ has a variational characterization by
the Rayleigh quotient:
\begin{multline}\label{rel}
\mu^1 = \inf\left\{ \int_{\omega}(A_{22}(X_2)\nabla u).\nabla u
  \big|\,u \in H_0^1(\omega) \text{ s.t. } \int_{\omega}u^2=1 \right\}\\
= \inf_{u \in H_0^1(\omega)\setminus \{0\}}  \frac{\int_{\omega}(A_{22}(X_2)\nabla u).\nabla u }{\int_{\omega}u^2}.
\end{multline}
Moreover, $W_1$ is simple and has constant sign in $\Omega$ (see
\cite{i}). The choice of positive sign leaves us with a unique $W_1$.

 Similarly, the eigenvalue problem \eqref{eq:5} has the following weak
 formulation: find $u\in  V(\Omega_\ell)\setminus\{0\}$ and a real
 number $\lambda$ such that
\begin{equation}\label{eigen1}
\int_{\Omega_\ell}A \nabla u.\nabla v \,dx = \lambda \int_{\Omega_\ell}u v\, dx\,,~~\forall v\in V(\Omega_\ell).
\end{equation}
It is  well known, see \cite{a}, that the first eigenvalue
$\lambda_{\ell}^1$ for \eqref{eigen1} is associated with a variational characterization,
\begin{equation}\label{reiligh}
\lambda_\ell^1 = \inf\left\{ \int_{\Omega_\ell}A\nabla u.\nabla u \,: \,
  u \in V(\Omega_\ell),\, \int_{\Omega_\ell}u^2 = 1
\right\}=\inf_{u \in V(\Omega_\ell)\setminus \{0\}}  \frac{\int_{\Omega_\ell}A(X_2)\nabla u.\nabla u }{\int_{\Omega_\ell}u^2}.
\end{equation}
It is also true, and can be proved in the same way as it is done for
the corresponding Dirichlet problem, that $\lambda_{\ell}^1$ is simple and the
corresponding eigenfunction $u_\ell$ has constant sign in
$\Omega_\ell$, that we should fix as the positive sign in the sequel. 
For some of our results we shall need to impose a certain symmetry
condition on $\omega$ and $A$. 
\begin{definition}
  \label{def:symmetry}
 We shall say that property (S) holds if  $\omega$ is symmetric w.r.t.~the origin (i.e., $-\omega=\omega$) and $A(-X_2)=A(X_2)$.
\end{definition}
From the uniqueness of $u_\ell$ we deduce easily the following
symmetry result.
\begin{proposition}\label{symmetry}
If property (S) holds  then 
 $u_\ell(x_1, X_2) = u_\ell(-x_1, -X_2)$.
\end{proposition}
\begin{proof}
Clearly $v_\ell(x_1, X_2) := u_\ell(-x_1,-X_2)$ is a positive
normalized eigenfunction for $\lambda^1_\ell$, so it must be equal to $u_\ell$.
\end{proof}

\section{The gap phenomenon in a model problem}\label{model}
In this section  we treat a two dimensional model problem in order to
illustrate the main ideas behind the analysis of the general case in
the next sections. Throughout this section $\omega=(-1,1)$,
$\Omega_\ell=(-\ell,\ell)\times(-1,1)$, and the matrix $A$ is a constant
matrix depending on the parameter $\delta\in[0,1)$, namely, 
\begin{equation}
%\begin{aligned}
\label{eq:111}
A = A_\delta =
\begin{pmatrix}
1 & \delta \\
\delta & 1
\end{pmatrix}.
%\end{aligned}
\end{equation}
Clearly $A_\delta$ satisfies all the assumptions made on $A$ in
Section~\ref{preliminaries}. Since the eigenvalues of $A_\delta$ are
$1\pm{}\delta$,  $\lambda_A = 1- \delta$ (see \eqref{e}). 
In  this section we shall denote a point in $\R^2$ by $x=(x_1, x_2).$
The  problem \eqref{eigen2} has the following simple form
\begin{equation*}\label{eqnforuinfty}
\left\{\begin{aligned}
-&W_1''= \mu{}^1 W_1 \text{ in } (-1, 1)\,,\\
&W_1(-1) =W_1(1) = 0\,.
\end{aligned}\right.
\end{equation*}
where $\mu{}^1$ denotes the first eigenvalue and $W_1$ is the corresponding
positive normalized eigenfunction. Therefore, $\mu{}^1=(\frac{\pi}{2})^2$ and $W_1(t)=\cos(\frac{\pi}{2}t)$.

\bigskip

\begin{proposition}\label{prop:iff}
For $\delta=0$ we have $\lambda^1_\ell = \mu{}^1$ for all $\ell>0$. For
$\delta\in(0,1)$ we have  
\begin{equation}
  \label{eq:2}
  (1-\delta^2)\mu{}^1<\lambda^1_\ell<\mu{}^1,\,\forall\ell>0.
\end{equation}
\end{proposition}
\begin{proof}
{\rm (i)} Since $A_0=\begin{pmatrix} 1 & 0\\
                              0 & 1
             \end{pmatrix}$,
 the corresponding operator is just $-\Delta$, and
 the function $v(x_1,x_2)=W_1(x_2)$ is clearly a  positive
 eigenfunction in \eqref{eq:5} with $\sigma=\mu{}^1$, for all
 $\ell>0$. It follows that $\lambda^1_\ell = \mu{}^1$ as claimed.\\
{\rm (ii)} Assume now that $\delta\in(0,1)$. Using the function
$v(x_1,x_2)=W_1(x_2)$ in the Rayleigh quotient \eqref{reiligh} yields
the inequality
\begin{equation}
  \label{eq:1}
  \lambda^1_\ell\leq \mu{}^1\,.
\end{equation}
 We claim that the inequality in \eqref{eq:1} is strict as stated in \eqref{eq:2}. Indeed, 
 an equality would imply that the function $v$ (as defined above) is a
 positive eigenfunction in \eqref{eq:5} for
 $\sigma=\lambda^1_\ell=\mu^1$, and in particular, it satisfies the
 Neumann boundary condition 
$$
0=(A_\delta\nabla v).\nu=v_{x_1}+\delta
 v_{x_2}=\delta v_{x_2}~\text{ on
 }~\Gamma_\ell^+=\{\ell\}\times(-1,1)\,.
$$
 But this clearly contradicts the fact that $(W_1)'(x_2)\neq0$ for
 $x_2\in(-1,1)\setminus\{0\}$.
 To prove the inequality of the left in \eqref{eq:2} we first notice the
 elementary inequality
 \begin{equation}
   \label{eq:7}
   (A_\delta\vec{\xi}).\vec{\xi} \geq (1 -\delta^2)|\xi_2|^2,~\forall\vec{\xi}=(\xi_1,\xi_2)\in\R^2\,.
 \end{equation}
Indeed, \eqref{eq:7} follows from the identity
\begin{equation}\label{eq:ident}
(A_\delta\vec{\xi}).\vec{\xi}=\xi_1^2+2\delta\xi_1\xi_2+\xi_2^2=(1-\delta^2)\xi_2^2+(\xi_1+\delta\xi_2)^2\,.
\end{equation}
By \eqref{eq:7}  and\eqref{rel}  we get 
\begin{equation}
\label{eq:8}
\begin{aligned}
\lambda_\ell^1 = \int_{\Omega_\ell} (A_\delta \nabla
u_\ell).\nabla u_\ell
&\geq (1 -\delta^2)\int_{\Omega_\ell}|\del_{x_2}u_\ell|^2 \\ &\geq (1 -\delta^2)\mu{}^1\int_{\Omega_\ell} |u_\ell|^2 
= (1 -\delta^2) \mu{}^1.
\end{aligned}
\end{equation}
 To conclude, we show that the inequality $\lambda_\ell^1\geq
 (1-\delta^2)\mu{}^1$ is strict. Indeed, equality would imply equalities
 in all the inequalities in \eqref{eq:8}, implying in particular that
 $u_\ell(x_1,x_2)=W_1(x_2)$ in $\Omega_\ell$. It would then follow
 that 
 $\lambda_\ell^1=\mu^1$. Contradiction.
\end{proof}

 From now on we shall assume that $\delta\in(0,1)$ (the first part of
 \rprop{prop:iff} settles completely the case $\delta=0$). Our main
 result in this section establishes the following estimate about the
 behavior of  $\lambda_\ell^1$ as $\ell$ goes to infinity.
\begin{theorem}\label{th:main-theorem}
$\limsup_{\ell \to \infty} \lambda_\ell^1 < \mu{}^1,$ for every $\delta\in(0,1)$.
\end{theorem}

In the next section, when dealing with the general case, we shall
actually see that the limit $\lim_{\ell \to \infty} \lambda_\ell^1$ exists. As mentioned in the Introduction, an important ingredient in the
proof  of \rth{th:main-theorem} is a study of the asymptotic behavior of
$\lambda_\ell^1 $ as $\ell \to 0$ (a dimension reduction problem).
\begin{theorem}\label{zero} We have
$\lim_{\ell \to 0} \lambda_\ell^1 = (1- \delta^2)\mu^1$.
\end{theorem}
\begin{proof}
 It suffices to consider $\ell<1$. Fix any $\alpha\in(0,1)$ and let
 $\rho_\ell$ be the piecewise-linear function defined by
 \begin{equation*}
   \rho_\ell(t)=
   \begin{cases}
     \frac{t+1}{\ell^\alpha} & t\in[-1,-1+\ell^\alpha),\\
     1 & t\in [-1+\ell^\alpha,1-\ell^\alpha],\\
    \frac{1-t}{\ell^\alpha} & t\in(1-\ell^\alpha,1]\,.
   \end{cases}
 \end{equation*}
Consider the following test function 
\begin{equation}
\label{eq:50}
v_{\ell}(x_1,\ x_2) =W_1(x_2) -\delta x_1W_1'(x_2)\rho_{\ell}(x_2)\,.
\end{equation}
 Then clearly $v_{\ell} \in V(\Omega_{\ell})$ is a valid test 
function. From  \eqref{reiligh},\ we have 
\begin{equation}\label{rel1}
\begin{aligned}
\lambda_{\ell}^1 \leq \frac{\int_{\Omega_{\ell}}A_{\delta}\nabla v_{\ell}.\nabla v_{\ell}}{\int_{\Omega_{\ell}}v_{\ell}^2} &= 
\frac{\int_{\Omega_{\ell}}|\del_{x_1}v_\ell|^2 + \int_{\Omega_{\ell}}|\del_{x_2}v_\ell|^2 + 2\delta \int_{\Omega_{\ell}}
\del_{x_1}v_\ell\del_{x_2}v_\ell}{\int_{\Omega_{\ell}}v_{\ell}^2}  \\ 
&= \frac{I_1 + I_2 + I_3}{I}\,.
\end{aligned}
\end{equation}
We consider each of the terms  $I_1, I_2, I_3$ and $ I$
separately.  First,  
\begin{eqnarray}\label{I_1}
I_1 = \delta^2 \int_{\Omega_{\ell}}\rho_{\ell}^2|W_1'(x_2)|^2\,dx = 2\ell\delta^2 \int_{-1}^1\rho_{\ell}^2|W_1'(x_2)|^2\,dx_2\,.
\end{eqnarray}
Next, calculating for $I_2$,\ 
\begin{align*}
I_2 &= \int_{\Omega_{\ell}} \Big[W_1'(x_2) - \delta x_1\{ \rho_{\ell}W_1''(x_2) + W_1'(x_2)\rho'_\ell(x_2)\}\Big]^2\\
&= \int_{\Omega_{\ell}} |W_1'|^2 -2\delta\int_{\Omega_{\ell}} x_1W_1'(x_2)\big\{ \rho_{\ell}W_1''(x_2) + W_1'(x_2)\rho'_\ell(x_2)\big\}
\\&\phantom{=} +\delta^2 \int_{\Omega_{\ell}}x_1^2| \rho_{\ell}W_1''(x_2) 
 + W_1'(x_2)\rho_\ell'(x_2)|^2.
\end{align*}
The integral in the middle vanishes since $\int_{-\ell}^{\ell}x_1 = 0$. Hence,
using  $|\rho_{\ell}'| \leq \frac{1}{\ell^{\alpha}}$ and \eqref{rel} we get 
\begin{equation}\label{I_2}
I_2 
=2\ell \mu{}^1  + \frac{2\delta^2\ell^3}{3} \int_{-1}^1 |\rho_{\ell}W_1'' + W_1'\rho'_\ell |^2 
\leq 2\ell \mu{}^1 + \frac{2\delta^2\ell^3}{3} (C_1 + C_2\ell^{-2\alpha}),
\end{equation}
where $C_1, C_2$ are two constants independent of $\ell$. Next,  for
$I_3$ we find,
%\textbf{Calculation for $I_3$}
\begin{multline}\label{Ii}
I_3 = 2\delta \int_{\Omega_{\ell}}-\delta W_1'\rho_{\ell}\left[W_1' -x_1\delta\left\{
W_1'\rho'_\ell + \rho_{\ell}W_1''\right\}\right]\\
=-4\ell\delta^2\int_{-1}^1\rho_{\ell}|W_1'|^2 + 2\delta^3 
\int_{\Omega_\ell}x_1 W_1'\rho_{\ell}\left\{ W_1'\rho'_\ell + \rho_{\ell}W_1''\right\}
= -4\ell\delta^2\int_{-1}^1\rho_{\ell}|W_1'|^2 .
\end{multline}
Finally we compute the term  $I$.
\begin{multline}\label{I}
I = \int_{\Omega_{\ell}}\left(W_1 - \delta x_1W_1'\rho_{\ell} \right)^2  
= \int_{\Omega_{\ell}}W_1^2 + \delta^2\int_{\Omega_{\ell}} x_1^2\rho_{\ell}^2|W_1'|^2 \\
= 2\ell  + \frac{2\ell^3\delta^2}{3}\int_{-1}^1\rho_{\ell}^2|W_1'|^2 \geq 2\ell.
\end{multline}
Plugging \eqref{I_1}--\eqref{I} in
\eqref{rel1} yields
\begin{equation}\label{subs}
\lambda_\ell^1 \leq \delta^2 \int_{-1}^1\rho_{\ell}^2|W_1'|^2 +  
\mu{}^1 -2\delta^2\int_{-1}^1\rho_{\ell}|W_1'|^2  + \varepsilon(\ell),
\end{equation}
where $\varepsilon(\ell) \to 0$ as $\ell \to 0$.  Since  $\rho_{\ell} \to 1$ \ pointwise, passing to the limit $\ell\to0$ and using
dominated convergence for the RHS of \eqref{subs} gives
\begin{equation}
\label{eq:9}
\limsup_{\ell \to 0} \lambda_{\ell}^1 \leq (1- \delta^2)\mu{}^1\,.
\end{equation}
Combining \eqref{eq:9} with \eqref{eq:2} we obtain the result of the
theorem. 
\end{proof}
\bigskip
Now we turn to the proof of \rth{th:main-theorem}.\  
\begin{proof}[Proof of \rth{th:main-theorem}]
% From the last theorem we have 
% \begin{equation}\label{convg}
% \frac{\int_{\Omega_{\ell}}A_{\delta}\nabla v_{\ell}.\nabla v_{\ell}}{\int_{\Omega_{\ell}}v_{\ell}^2} = (1-\delta^2)\mu{}_{\infty} + \varepsilon(\ell),
% \end{equation}
% where $\varepsilon(\ell) \to 0$,\  as $\ell \to 0.$
 Let $\ell_0$ and $\eta$ be two positive constants whose values will be
 determined later. For $\ell>\ell_0+\eta$ define $\phi_\ell$ by
$$
 \phi_\ell = \begin{cases}
v_{\ell_0}(x_1-\ell +\ell_0, x_2) & \text{ on }  (\ell-\ell_0, \ell)\times{}(-1,1) \,,\\
\frac{\left(x_1- (\ell -\ell_0-\eta)\right)W_1(x_2)}{\eta}  &\text{ on }  (\ell -\ell_0 -\eta, \ell-\ell_0)\times{}(-1, 1)\,, \\
 0  & \text{ on }  \Omega_{\ell -\ell_0-\eta}\,, \\
  \frac{\left(-x_1- (\ell -\ell_0-\eta)\right)W_1(x_2)}{\eta} &  \text{ on }  \left(\ell_0 -\ell, -\ell + \ell_0 + \eta\right)\times(-1, 1)\,, \\
  v_{\ell_0}(x_1 + \ell -\ell_0,  x_2) & \text{ on }  (-\ell, \ell_0 - \ell)\times(-1,1)\,,
\end{cases}
$$
 where $v_{\ell_0}$ is given by \eqref{eq:50}.
 We have
\begin{equation}
\label{squarenorm}
 \begin{aligned}
\int_{\Omega_\ell}\phi_\ell^2 &= \int_{\Omega_\ell\setminus\Omega_{\ell-\ell_0}}\phi_\ell^2 +
\int_{\Omega_{\ell-\ell_0}}\phi_\ell^2 \\ &=
\int_{\Omega_{\ell_0}}v_{\ell_0}^2+2\left(\int_{\ell-\ell_0-\eta}^{\ell-\ell_0}\frac{(x_1-\ell
    + \ell_0 + \eta)^2}{\eta^2}dx_1\right)\left(\int_{-1}^1W_1^2\right) 
\\& = \int_{\Omega_{\ell_0}}v_{\ell_0}^2+ \frac{2}{3}\eta\,,
\end{aligned}
\end{equation}
where we used the fact that $\phi_\ell$ is an even function in $x_1$ on
$\Omega_\ell\setminus\Omega_{\ell-\ell_0}$. 
Also, 
\begin{equation}
\label{eq:73}
\int_{\Omega_\ell}A_\delta\nabla\phi_\ell.\nabla\phi_\ell = \int_{\Omega_{\ell_0}}A_\delta\nabla v_{\ell_0}.\nabla v_{\ell_0}
+ \int_{\Omega_{\ell-\ell_0}}A_\delta \nabla \phi_\ell.\nabla\phi_\ell\,.
\end{equation}
Setting $\mathcal{D} = \Omega_{\ell-\ell_0}\setminus \Omega_{\ell-\ell_0-\eta}$ and using the fact that 
$\phi_\ell$ is even in $\mathcal{D}$ while $\del_{x_1}\phi_\ell$ is odd on
$\mathcal{D}$ we get
\begin{equation}
\label{num}
\begin{aligned}
\int_{\Omega_{\ell-\ell_0}}A_\delta \nabla \phi_\ell.\nabla\phi_\ell &= \frac{1}{\eta^2}\int_{\mathcal{D}}W_1^2 
+2\delta \int_{\mathcal{D}}\del_{x_1}\phi_\ell \del_{x_2}\phi_\ell \\ 
 &\phantom{=}+ \frac{2}{\eta^2} \int_{(\ell-\ell_0-\eta,\ell-\ell_0)\times(-1,1)}|W_1'|^2(x_1-\ell+\ell_0
 +\eta)^2\\
&=  \frac{2}{\eta}\int_{-1}^1W_1^2 +\frac{2\eta}{3}\int_{-1}^1|W_1'|^2 
= \frac{2}{\eta} + \frac{2\eta \mu^1}{3}\,. 
\end{aligned}
\end{equation}
From \eqref{squarenorm}--\eqref{num} we obtain
\begin{equation}
\label{eq:10}
\lambda_\ell^1 \leq \frac{\int_{\Omega_{\ell_0}}A_\delta\nabla v_{\ell_0}.\nabla v_{\ell_0} +\frac{2}{\eta} + \frac{2\eta\mu{}^1}{3} }{\int_{\Omega_{\ell_0}}v_{\ell_0}^2 + \frac{2}{3}\eta}\,. 
\end{equation}
Noting that \rth{zero} implies that
\begin{equation*}
  \frac{\int_{\Omega_{\ell_0}}A_{\delta}\nabla v_{\ell_0}.\nabla v_{\ell_0}}{\int_{\Omega_{\ell_0}}v_{\ell_0}^2} = (1-\delta^2)\mu{}^1 + \varepsilon(\ell_0)\,,
\end{equation*}
we obtain from \eqref{eq:10} that 
\begin{equation}
\label{eq:11}
\begin{aligned}
\lambda_\ell^1 -\mu^1&\leq \frac{\left\{(1-\delta^2)\mu{}^1 +
    \varepsilon(\ell_0)\right\}\int_{\Omega_{\ell_0}}v_{\ell_0}^2 +\frac{2}{\eta} +
  \frac{2\eta\mu{}^1}{3}  }{\int_{\Omega_{\ell_0}}v_{\ell_0}^2 +
  \frac{2}{3}\eta}-\mu^1\\
&= \frac{(\varepsilon(\ell_0)-\delta^2\mu{}^1)\int_{\Omega_{\ell_0}}v_{\ell_0}^2 + \frac{2}{\eta} }{\int_{\Omega_{\ell_0}}v_{\ell_0}^2 + \frac{2}{3}\eta}\,.
\end{aligned}
\end{equation}
Choosing  $\ell_0$ small enough such that $\varepsilon(\ell_0)-\delta^2\mu{}^1 < 0$, and then
taking $\eta $ sufficiently large, makes the RHS of \eqref{eq:11} equal a 
negative number, say $-\delta_0$. Hence,
$\lambda_\ell^1\leq \mu^1-\delta_0$ for $\ell>\ell_0+\eta$, and the result follows.
\end{proof}

\section{The gap phenomenon in the general case.}\label{ge}
In this section we extend the results  from Section~\ref{model} to a more general
framework. We shall use the notation from
Section~\ref{preliminaries} and study the limit $\lim_{\ell\to\infty} \lambda_\ell^1$ for
$\lambda_\ell^1$ given by \eqref{reiligh}. As in Section~\ref{model} our
strategy is to study first the limit as $\ell$ goes to $0$.
\begin{theorem}
  \label{th:lim-zero}
 We have $\lim_{\ell\to0} \lambda_\ell^1=\Lambda^1$ where
 \begin{equation}
   \label{eq:Lambda}
\Lambda^1=\inf\left\{ \int_{\omega}A_{22}(X_2)\nabla u.\nabla
  u-\frac{|A_{12}(X_2).\nabla u|^2}{a_{11}(X_2)}
  :\,u \in H_0^1(\omega),\, \int_{\omega}u^2=1 \right\}.
 \end{equation}
\end{theorem}
\begin{proof}
  The reason why we find $\Lambda^1$ as the
  limiting value  will be clarified by  the
  following simple observation. Let 
  $B=\begin{pmatrix}
b_{11} & B_{12} \\
B^t_{12} & B_{22}
\end{pmatrix}$ be a positive definite $n\times n$ matrix and represent any
vector $\vec{z}$ in $\R^n$ as $\vec{z}=(z_1,Z_2)$ with $Z_2\in \R^{n-1}$. Then,
elementary calculus shows that for any fixed $Z_2\in\R^{n-1}$ we have
  \begin{equation}
\label{eq:obs}
    \min_{z_1\in\R} (B\vec{z}).\vec{z}=(B_{22}Z_2).Z_2-\frac{|B_{12}Z_2|^2}{b_{11}}\,.
  \end{equation}
Furthermore, the minimum in \eqref{eq:obs} is attained for
\begin{equation}
\label{eq:13}
z_1=-\frac{B_{12}Z_2}{b_{11}}\,.
\end{equation}
 Applying \eqref{eq:obs} with $B=A(X_2)$ we obtain, for any $\ell>0$,
 \begin{equation}
\label{eq:12}
\begin{aligned}
   \int_{\Omega_\ell} (A(X_2)\nabla u_\ell).\nabla u_\ell&\geq  \int_{\Omega_\ell}
   (A_{22}(X_2)\nabla_{X_2}u_\ell).\nabla_{X_2}u_\ell-\frac{|A_{12}(X_2)\nabla_{X_2}u_\ell|^2}{a_{11}(X_2)}\\
&\geq  \Lambda^1\int_{\Omega_\ell}u_\ell^2\,.
\end{aligned} 
 \end{equation}
 By \eqref{eq:12} the lower-bound 
 \begin{equation}
 \label{eq:liminf}
  \liminf_{\ell\to0} \lambda_\ell^1\geq \Lambda^1\,,
 \end{equation} 
 is clear.
We note that from the above it follows in particular that 

$$\Lambda^1\geq\lambda_A\cdot \inf\left\{ \int_{\omega}|\nabla u|^2
   :\,u \in H_0^1(\omega),\, \int_{\omega}u^2=1 \right\}.$$
(see \eqref{e}) and the infimum in \eqref{eq:Lambda}
is actually a minimum, which is realized by a positive function $w_1\in H^1_0(\omega)$.

In order to complete the proof of \rth{th:lim-zero} we need to establish  the upper-bound part.  A natural
generalization of the construction used in the proof of \rth{zero}
would be to use 
\begin{equation}
  \label{eq:14}
  v_\ell(x)=w_1(X_2)-\frac{\left( A_{12}(X_2).\nabla w_1\right) x_1 \rho_\ell(X_2)}{a_{11}(X_2)}\,,
\end{equation}
 where $\rho_l$ is an appropriate cut-off
function. However, since the 
coefficients of the matrix $A(X_2)$ are only assumed to be $L^\infty$-functions, the function
on the RHS of \eqref{eq:14} does not necessarily belong to $H^1$. To
overcome this difficulty, we use an approximation argument, motivated
by \cite[Ch.~14]{braides}. We apply standard mollification to define
 a family  of functions $\{G_\varepsilon\}_{\varepsilon>0}\subset C^\infty_c(\omega)$ satisfying
\begin{equation}
  \label{eq:15}
 \lim_{\varepsilon\to0} G_\varepsilon(X_2)=\frac{A_{12}(X_2)\cdot \nabla w_1}{a_{11}(X_2)}~\text{ in
 }L^2(\omega)\text{ and a.e..}
\end{equation}
We then define 
\begin{equation}
\label{eq:18}
v_\ell^\varepsilon (x_1,X_2)= w_1(X_2) -G_\varepsilon(X_2)x_1\,.  
\end{equation}
First notice that
\begin{equation}
\label{eq:16}
\int_{\Omega_\ell}|v^\varepsilon_\ell|^2  = \int_{-\ell}^\ell\int_{\omega}w_1^2 - 2x_1w_1G_\varepsilon+ \left( x_1G_\varepsilon\right)^2
\geq 2\ell \int_{\omega}w_1^2= 2\ell,
\end{equation}
since $\int_{-\ell}^{\ell}x_1\,dx_1 = 0.$ Now
\begin{align*}
\int_{\Omega_\ell}A\nabla v_\ell^\varepsilon.\nabla v_\ell^\varepsilon &= \int_{\Omega_\ell} a_{11}(\partial_{x_1} v_\ell^\varepsilon)^2 +  2
(A_{12}.\nabla_{X_2}v_\ell^\varepsilon) \partial_{x_1}v_\ell^\varepsilon + (A_{22}\nabla_{X_2} v_\ell^\varepsilon).\nabla_{X_2} v_\ell^\varepsilon \\
&= I_1(\varepsilon) + I_2(\varepsilon) + I_3(\varepsilon)\,.
\end{align*}
For the first integral we have 
\begin{equation}\label{firs}
I_1 (\varepsilon)= \int_{\Omega_\ell}a_{11} G_\varepsilon^2  = 2\ell \int_\omega a_{11} G_\varepsilon^2\,.
\end{equation}
For the second integral,
 \begin{equation}
\label{eq:69}
 I_2(\varepsilon) = 
 2\int_{-\ell}^\ell \int_{\omega} A_{12}.\Big\{\nabla w_1 -x_1\nabla 
   G_\varepsilon(X_2) \Big\} \Big\{-G_\varepsilon(X_2)\Big\}\,.
 \end{equation}
 Since the integral of the term containing  $x_1$ vanishes, we get 
\begin{equation}\label{seco}
I_2 (\varepsilon)= -4\ell\int_\omega(A_{12}. \nabla w_1)G_\varepsilon\,.
\end{equation}
For the last integral we have (after dropping the term with the vanishing integral),
\begin{multline}
\label{thrd}
I_3(\varepsilon) =  \int_{-\ell}^\ell \int_{\omega}
(A_{22}\nabla w_1).\nabla w_1+ x_1^2(A_{22}\nabla G_\varepsilon) .\nabla G_\varepsilon \\=2\ell\Big\{
\int_\omega (A_{22}\nabla w_1).\nabla w_1+\frac{\ell^2}{3}\int_\omega
(A_{22}\nabla G_\varepsilon).\nabla G_\varepsilon \Big\} \,.
\end{multline}
By \eqref{eq:16}--\eqref{thrd} we deduce that
\begin{multline}
\label{eq:17}
          \limsup_{\ell\to0} \lambda_\ell^1\leq 
          \limsup_{\ell\to0} \frac{\int_{\Omega_\ell} A\nabla v_\ell^\varepsilon.\nabla v_\ell^\varepsilon}{\int_{\Omega_\ell} |v_\ell^\varepsilon|^2}\leq\\    
          \int_\omega a_{11}\left( G_\varepsilon\right)^2-2\int_\omega(A_{12}. \nabla w_1)G_\varepsilon
          +\int_\omega \big(A_{22}\nabla w_1\big).\nabla w_1\,.
   \end{multline} 
   Passing to the limit $\varepsilon\to0$ in \eqref{eq:17}, using \eqref{eq:15}, gives
 \begin{equation*}
    \limsup_{\ell\to0} \lambda_\ell^1\leq 
    \int_{\omega}(A_{22}\nabla  w_1).\nabla 
  w_1-\frac{|A_{12}\nabla  w_1|^2}{a_{11}}=\Lambda^1\,,
 \end{equation*}
  which together with \eqref{eq:liminf} yields the result.
\end{proof}
\begin{remark}
  \label{rem:Ww}
 Replacing \eqref{eq:18} by 
 \begin{equation}
   \label{eq:22}
 \tilde v_\ell^\varepsilon (x_1,X_2)= W_1(X_2) -\tilde G_\varepsilon(X_2)x_1\,, 
 \end{equation}
where $\tilde G_\varepsilon$ is defined as in  \eqref{eq:15}, but with $w_1$
replacing   $W_1$, and carrying out  the same computation as in the last part of the proof of
 \rth{th:lim-zero} yields
 \begin{equation}
   \label{eq:23}
\inf_{\varepsilon>0} \lim_{\ell\to0}\frac{\int_{\Omega_\ell} (A\nabla \tilde v_\ell^\varepsilon).
  \nabla \tilde v_\ell^\varepsilon}{\int_{\Omega_\ell}|\tilde v_\ell^\varepsilon|^2}=\int_{\omega}(A_{22}\nabla  W_1).\nabla 
  W_1-\frac{|A_{12}\nabla  W_1|^2}{a_{11}}\,.
 \end{equation}
\end{remark}
Our next theorem provides an analog of \rth{th:main-theorem} to the general case.  
\begin{theorem}\label{th1}
We have 
\begin{equation}\label{res} 
\limsup_{\ell \to \infty }  \lambda_\ell^1 < \mu^1,
\end{equation}
 provided the following condition holds,
\begin{equation}\label{con}
A_{12}.\nabla W_1\not\equiv 0\text{ a.e. on } \omega.
\end{equation}
 In case \eqref{con} does not hold we have $\lambda_\ell^1 = \mu^1$ for all $\ell>0$.
\end{theorem}
\begin{remark}
  It is easy to construct examples where condition \eqref{con} doesn't
  hold. Take for example for $\omega$ the unit disc in $\R^2$. For $A_{22}=\begin{pmatrix} 1 & 0\\
                              0 & 1
             \end{pmatrix}$, 
  the eigenfunction $W_1$ is radially symmetric. We use polar
  coordinates on $\omega$ and represent each $X_2$ as
  $X_2=r(\cos\theta,\sin\theta)$. Taking $a_{11}=1$ and 
  $A_{12}(X_2)=t(-\sin \theta,\cos \theta)$ for $|t|$ small enough
  (in order for the uniform ellipticity condition \eqref{e} to hold
  for the $3$ by $3$ matrix $A$)
  yields  an example for which \eqref{con} doesn't hold. 
\end{remark}
\begin{proof}
{\rm (i)} Assume first that \eqref{con} holds. Then,
\begin{equation}
\label{eq:Lam<mu}
 \Lambda_1<\mu{}^1.
\end{equation}
Indeed, this follows from
\begin{equation*}
 \Lambda_1\leq \int_{\omega}A_{22}(X_2)\nabla W_1.\nabla
  W_1-\frac{|A_{12}(X_2)\nabla W_1|^2}{a_{11}(X_2)}<\int_{\omega}A_{22}(X_2)\nabla W_1.\nabla
  W_1=\mu{}^1\,.
\end{equation*}
By the proof of \rth{th:lim-zero} there exist  positive values of
$\ell_0$ and $\varepsilon_0$ such that $\tilde v_{\ell_0}^{\varepsilon_0}$ defined by \eqref{eq:22}
satisfies
\begin{equation}
  \label{eq:19}
\int_{\Omega_{\ell_0}} A\nabla\tilde
v_{\ell_0}^{\varepsilon_0}.\nabla\tilde
v_{\ell_0}^{\varepsilon_0}<\mu{}^1\int_{\Omega_{\ell_0}}|\tilde v_{\ell_0}^{\varepsilon_0}|^2\,.
\end{equation}
Notice  that $\tilde v_{\ell_0}^{\varepsilon_0}(0,X_2)=W_1(X_2)$. Let $\eta>0$ be a
parameter whose value will be determined later. For $\ell>\ell_0+\eta$ define $\phi_\ell$ as follows,
\begin{equation}
\label{eq:38}
 \phi_\ell = 
 \begin{cases}
 \tilde v_{\ell_0}^{\varepsilon_0}(x_1-\ell +\ell_0,  X_2)  &\text{ on }  (\ell-\ell_0,
 \ell)\times\omega \,, \\
\frac{\left(x_1- (\ell -\ell_0-\eta)\right) W_1(X_2)}{\eta}  &\text{ on }  (\ell -\ell_0 -\eta, \ell-\ell_0)\times\omega \,,\\
 0   &\text{ on }  \Omega_{\ell -\ell_0-\eta}\,, \\
  \frac{\left(-x_1- (\ell -\ell_0-\eta)\right)W_1(X_2)}{\eta}  & \text{ on }  \left(\ell_0 -\ell, -(\ell-\ell_0-\eta)\right)\times\omega \,,\\
  \tilde v_{\ell_0}^{\varepsilon_0}(x_1 + \ell -\ell_0,  X_2) &\text{ on }  (-\ell, \ell_0 - \ell)\times\omega\,.
\end{cases}
\end{equation}
 Since 
\begin{equation*}
  \int_{\Omega_\ell\setminus\Omega_{\ell-\ell_0}}\phi_\ell^2= \int_{\Omega_{\ell_0}}|\tilde v_{\ell_0}^{\varepsilon_0}|^2\,,
\end{equation*}
and
\begin{equation*}
   \int_{\Omega_{\ell-\ell_0}}\phi_\ell^2=2\left(\int_{\ell-\ell_0-\eta}^{\ell-\ell_0}\frac{(x_1-\ell + \ell_0 + \eta)^2}{\eta^2}dx_1\right)\left(\int_{\omega}W_1^2\,dX_2\right) = \frac{2}{3}\eta\,,
\end{equation*}
we have
\begin{equation}\label{squarenorm2}
\int_{\Omega_\ell}\phi_\ell^2 = \int_{\Omega_{\ell_0}}|\tilde v_{\ell_0}^{\varepsilon_0}|^2 + \frac{2}{3}\eta\,.
\end{equation} 
Similarly
\begin{equation}
\label{eq:20}
\int_{\Omega_\ell}A\nabla\phi_\ell.\nabla\phi_\ell = \int_{\Omega_{\ell_0}}A\nabla  \tilde v_{\ell_0}^{\varepsilon_0}.\nabla  \tilde v_{\ell_0}^{\varepsilon_0}
+ \int_{\Omega_{\ell-\ell_0}}A\nabla \phi_\ell.\nabla\phi_\ell\, .
\end{equation}
 Setting $D = \Omega_{\ell-\ell_0}\setminus \Omega_{\ell-\ell_0-\eta}$ and 
$D^{+} = (\ell - \ell_0 -\eta,  \ell-\ell_0)\times{}\omega$, the last integral above can be written as
\begin{multline*}
\int_{\Omega_{\ell-\ell_0}}A\nabla \phi_\ell.\nabla\phi_\ell = \frac{1} {\eta^2}\int_{D}a_{11}W_1^2+ 
2\int_{D}\big(A_{12}.\nabla_{X_2}\phi_\ell\big) \del_{x_1} \phi_\ell  \\
 + \frac{2} {\eta^2}\int_{D^+} (x_1-\ell  +\ell_0 +\eta)^2A_{22}\nabla W_1.\nabla W_1.
\end{multline*}
The second integral vanishes since its integrand is an odd function of
$x_1$ on $D$. Therefore,
\begin{equation}
\label{eq:21}
\int_{\Omega_{\ell-\ell_0}}A\nabla \phi_\ell.\nabla\phi_\ell = \frac{2}{\eta} \int_{\omega} a_{11} W_1^2 + \frac{2\eta}{3}
\int_{\omega} A_{22}\nabla W_1.\nabla W_1 
= \frac{2}{\eta} \int_{\omega} a_{11} W_1^2 + \frac{2\eta\mu{}^1}{3}\,.
\end{equation}
Combining \eqref{squarenorm2}, \eqref{eq:20} and \eqref{eq:21}  we obtain
\begin{equation*}
\lambda_\ell^1 \leq  \frac{\int_{\Omega_\ell}A\nabla  \phi_\ell.\nabla \phi_\ell}{\int_{\Omega_\ell}\phi_\ell^2}
\leq \frac{\int_{\Omega_{\ell_0}}A\nabla  \tilde v_{\ell_0}^{\varepsilon_0}.\nabla \tilde v_{\ell_0}^{\varepsilon_0} +\frac{2}{\eta} \int_{\omega} a_{11}W_1^2 + 
\frac{2}{3}\eta\mu{}^1}{\int_{\Omega_{\ell_0}}|\tilde v_{\ell_0}^{\varepsilon_0}|^2+ \frac{2}{3}\eta }\,.
\end{equation*}
Therefore,
\begin{equation}
\label{eq:24}
\lambda_\ell^1 - \mu^1 \leq   \frac{      \int_{\Omega_{\ell_0}}A\nabla  \tilde
    v_{\ell_0}^{\varepsilon_0}.\nabla \tilde v_{\ell_0}^{\varepsilon_0}-\mu^1 \int_{\Omega_{\ell_0}}|\tilde
    v_{\ell_0}^{\varepsilon_0}|^2   + \frac{2}{\eta} \int_{\omega}
 a_{11}W_1^2}{\int_{\Omega_{\ell_0}}|\tilde v_{\ell_0}^{\varepsilon_0}|^2 + \frac{2}{3}\eta }.
\end{equation}
By \eqref{eq:19} it is clear that we can fix a large enough value for
$\eta$ such that the RHS of \eqref{eq:24} is negative, and the result for
case (i) follows.\\[2mm]
(ii)  By \eqref{eq:12} we have $\Lambda_1\leq
\lambda_\ell^1$ for all $\ell>0$. On the other hand, using $u(x)=W_1(X_2)$ as a
test function in \eqref{reiligh} gives $\lambda_\ell^1\leq \mu{}^1$. Thus we have,
\begin{equation}
  \label{eq:25}
 \Lambda_1\leq\lambda_\ell^1\leq \mu{}^1\,,\quad \forall\ell>0\,.
\end{equation}
In view of \eqref{eq:25}, the result for the case where
\eqref{con} doesn't hold would follow once we show that in this case  $\Lambda_1=\mu^1$. The Euler-Lagrange equation for
an eigenfunction $v$ of the quadratic  form in \eqref{eq:Lambda},
with eigenvalue $\lambda$  is
\begin{equation}
  \label{eq:26}
\left\{ 
\begin{aligned}
 -&\Div (A_{22}\nabla v)+\Div((A_{12}.\nabla v)A_{12}^t/{a_{11}})=\lambda v~\text{ in }\omega\,,\\
&v=0~\text{ on }\partial\omega\,.
\end{aligned}
\right.
\end{equation}
 Of course $v=w_1$ satisfies \eqref{eq:26} with
 $\lambda=\Lambda_1$. But since we assume that
  \eqref{con} doesn't hold, $v=W_1$ is also a solution of \eqref{eq:26} with
 $\lambda=\mu^1$. However, only the first  eigenvalue of the
 problem \eqref{eq:26} can  have a positive eigenfunction, so we must have $\Lambda_1=\mu^1$ as claimed.
\end{proof}

\section{Characterization of the limit $\lim_{\ell\to\infty} \lambda^1_\ell$}
\label{sec:lim}
In this section we obtain more precise results on the asymptotic
behavior of the eigenfunctions$\{u_\ell\}$ and the eigenvalues
$\{\lambda_\ell^1\}$ as $\ell$ goes to infinity. We shall see that when \eqref{con}
holds, the eigenfunctions decay to zero in the bulk of the cylinder
and concentration occurs near the bases of the cylinder.  We denote by  $[x]$  the  integer part of $x$. 
\begin{theorem}\label{th:decay}
Assume \eqref{con} holds. Then, there exist  $\alpha\in(0,1)$ and a positive constant $c$ such that for
$\ell>\ell_0$ we have, for every $0<r\leq \ell-1$,
\begin{align}
\label{eq:27}
\int_{\Omega_r} u_\ell^2 &\leq \alpha^{[\ell-r]}\,,\\
 \shortintertext{and}
   \label{eq:32}
\int_{\Omega_{r}} |\nabla u_\ell|^2&\leq c \alpha^{[\ell-r]}\,.
 \end{align}
\end{theorem}
\begin{proof}
Let $\ell$ and $\ell'$ satisfy  $0< \ell^{'} \leq \ell-1$. Define $\rho_{\ell^{'}}=\rho_{\ell^{'}}(x_1)$ by
\begin{equation}
  \label{eq:28}
\rho_{\ell^{'}}(x_1)=
\begin{cases}
   1 & |x_1|\leq \ell'\,,\\
\ell'+1-|x_1|& |x_1|\in (\ell',\ell'+1)\,,\\
   0 & |x_1|\geq \ell'+1\,.
   \end{cases}
\end{equation}
 Using  $v =\rho_{\ell^{'}}^2u_\ell \in V(\Omega_\ell)$
in \eqref{eigen1}, we get
\begin{equation*}
\int_{\Omega_\ell} (A\nabla u_\ell). \nabla (\rho_{\ell^{'}}^2u_\ell) = \lambda_\ell^1 \int_{\Omega_\ell}\rho_{\ell^{'}}^2u_\ell^2\,,
\end{equation*}
i.e.,
\begin{equation} \label{ll}
 \int_{\Omega_\ell} \big(A \nabla (\rho_{\ell^{'}}u_\ell)\big).\nabla (\rho_{\ell^{'}}u_\ell) - 
\int_{\Omega_\ell}u_\ell^2 (A\nabla \rho_{\ell^{'}}).\nabla \rho_{\ell^{'}} = \lambda_\ell^1\int_{\Omega_\ell}\rho_{\ell^{'}}^2u_\ell^2\,.
\end{equation}
Since $\rho_{\ell^{'}} u_\ell \in H_0^1(\Omega_\ell)$, by the Rayleigh quotient
characterization  of $\sigma_\ell^1$ (see \eqref{eq:4}) we have
\begin{equation}\label{diri}
\sigma_\ell^1  \int_{\Omega_\ell}u_\ell^2 \rho_{\ell^{'}}^2 \leq 
\int_{\Omega_\ell} A\nabla(\rho_{\ell^{'}}u_\ell).\nabla (\rho_{\ell^{'}}u_\ell)\,.
\end{equation}
Combining \eqref{ll}--\eqref{diri} with \eqref{mod} we get
\begin{equation}\label{zero1}
\begin{aligned}
(\sigma_\ell^1 -\lambda_\ell^1) \int_{\Omega_\ell} u_\ell^2 \rho_{\ell^{'}}^2 \leq
\int_{\Omega_\ell}u_\ell^2 (A\nabla \rho_{\ell^{'}}).\nabla \rho_{\ell^{'}}&= \int_{\Omega_{\ell'+1}\setminus \Omega_{\ell'}}u_\ell^2 (A\nabla
\rho_{\ell^{'}}).\nabla \rho_{\ell^{'}}\\
&\leq C_A  \int_{\Omega_{\ell'+1}\setminus \Omega_{\ell'}}u_\ell^2\,.
\end{aligned}
\end{equation}
By \eqref{upplow} and \eqref{res} there exists $\beta>0$ such that for
$\ell>\ell_0$ we have $\sigma_\ell^1 -\lambda_\ell^1\geq \beta$. Therefore, from \eqref{zero1} we
deduce that 
\begin{equation*}
(C_A+\beta)\int_{\Omega_{\ell^{'}}} u_\ell^2  \leq C_A \int_{\Omega_{\ell^{'}+1}}u_\ell^2\,.
\end{equation*}
This leads to 
\begin{equation}\label{iteration}
\int_{\Omega_{\ell^{'}}} u_\ell^2 \leq \alpha\int_{\Omega_{\ell^{'} + 1}} u_\ell^2\,,
\end{equation}
 with  $\alpha=\frac{C_A}{C_A+\beta}<1$.  Applying \eqref{iteration} successively for $\ell^{'} = r, r+1, \ldots,r+ [\ell-r]-1$ yields
\begin{equation}
\label{eq:29}
\int_{\Omega_r} u_\ell^2 \leq \alpha^{[\ell-r]} \int_{\Omega_\ell}u_\ell^2  =\alpha^{[\ell-r]}  \,.
\end{equation}
To prove \eqref{eq:32}, we fix $r\in(0,\ell-2)$ and then use
\eqref{ll}, with $\ell'=r$, combined with 
\eqref{e} and \eqref{eq:1}, to obtain
\begin{multline}
\label{eq:30}
\lambda_A \int_{\Omega_r}|\nabla u_\ell| ^2\leq \int_{\Omega_\ell} A \nabla (
\rho_{r}u_\ell).\nabla ( \rho_{r}u_\ell)\\ = 
\int_{\Omega_\ell}u_\ell^2 (A\nabla  \rho_{r}).\nabla \rho_{r} +
\lambda_\ell^1\int_{\Omega_\ell} \rho_{r}^2u_\ell^2\leq (C_A+\mu^1)\int_{\Omega_{r+1}}u_\ell^2\,.
\end{multline}
Finally, \eqref{eq:32} follows from \eqref{eq:29}--\eqref{eq:30} for $r\leq \ell -2$. Choosing a step size of $\frac 12$ in the first part of the proof would allow $r\leq \ell -1$. 
\end{proof}

The decay of the eigenfunction in the bulk immediately implies concentration near the two ends of the cylinder.
\begin{corollary}
 If \eqref{con} holds then for every $r\in(0,\ell-1]$ we have
\begin{equation}
\label{eq:31}
\int_{\Omega_\ell\setminus \Omega_{r}} u_\ell^2 \geq
1-\alpha^{[\ell-r]}~\text{ and }~\int_{\Omega_\ell\setminus \Omega_{r}}
\!\!\!A\nabla u_\ell.\nabla u_\ell\geq \lambda_\ell^1-c_1\alpha^{[\ell-r]}\,.
\end{equation}
\end{corollary}

To have a more precise description of  the asymptotic behavior of
$\lambda_\ell^1$ we introduce
two variational problems on  semi-infinite cylinders. Set 
\begin{equation*}
  \Omega_{\infty}^+ =  (0, \infty) \times\omega ~\text{ and }~ \Omega_{\infty}^- =  (-\infty,0) \times\omega\,,
\end{equation*}
 and denote the corresponding  lateral parts of the boundary by
 \begin{equation*}
   \gamma_\infty^{+} = (0, \infty)\times\partial\omega ~\text{ and }~   \gamma_\infty^{-} = (-\infty,0)\times\partial\omega\,.
 \end{equation*}
Define the spaces 
\begin{equation*}
V(\Omega_{\infty}^\pm) :=  \{ u\in H^1(\Omega_{\infty}^\pm)\,: \, u=0 \text{ on } \gamma_\infty^{\pm} \}\,,
\end{equation*}
 and set
\begin{equation}
\label{eq:33}
\nu_\infty^\pm= \inf_{0\neq u \in V(\Omega_{\infty}^\pm )} \frac{\int_{\Omega_\infty^\pm} A\nabla u.\nabla
  u}{\int_{\Omega_\infty^\pm} u^2}\,.
\end{equation}
\begin{remark}
\label{rem:reflection}
  In case property (S) holds (see Definition~\ref{def:symmetry}) we
  clearly have $\nu_\infty^+=\nu_\infty^-$ as we can use the transformation
  $v(x_1,X_2)\mapsto v(-x_1,-X_2)$ to pass from a function in $V(\Omega_{\infty}^+ )$
  to a function in $V(\Omega_{\infty}^-)$  (and vice versa) that has the same
  Rayleigh quotient. 
 In general we can only assert that $\nu_\infty^-=\tilde\nu_\infty^+$ where
 $\tilde\nu_\infty^+$ is defined as in \eqref{eq:33}, but with $A$ being
 replaced by  $\tilde A$, given by
   \begin{equation*}
\begin{aligned}
  \tilde A(X_2)=
\begin{pmatrix}
    a_{11}(X_2)  & -A_{12}(X_2) \\
-A_{12}^t(X_2) & A_{22}(X_2)
  \end{pmatrix}\,.\end{aligned}
 \end{equation*}
 This is easily seen by applying the transformation $v(x_1,X_2)\mapsto v(-x_1,X_2)$.
\end{remark}
The next lemma gives the possible range of values for
$\nu_\infty^\pm$.
\begin{lemma}
  \label{lem:nui}
 We have
 \begin{equation}
   \label{eq:63}
   0<\nu_\infty^\pm\leq \mu^1\,.
 \end{equation}
\end{lemma}
\begin{proof}
   By \rrem{rem:reflection} it is enough to consider
   $\nu_\infty^+$. The fact that $\nu_\infty^+>0$ follows from the 
   Poincar\'e inequality. In order to show that $\nu_\infty^+\leq \mu^1$
 we set for each $\varepsilon>0$,
 \begin{equation*}
   v_\varepsilon(x)=e^{-\varepsilon x_1} W_1(X_2)\,.
 \end{equation*}
Clearly $v_\varepsilon\in V(\Omega_\infty^+)$ and a direct computation
gives
\begin{equation}
\label{eq:51}
\begin{aligned}
   \int_{\Omega_\infty^+} A\nabla v_\varepsilon.\nabla
   v_\varepsilon&=
\int_{\Omega_\infty^+}\!\! e^{-2\varepsilon
x_1}\Big(a_{11}\varepsilon^2W_1^2-2\varepsilon(A_{12}.\nabla W_1)W_1+
A_{22}\nabla W_1.\nabla W_1\Big)\\
&=(\int_0^\infty\!\! e^{-2\varepsilon x_1})\Big(\mu^1+\varepsilon^2 \int_\omega
a_{11}W_1^2-2\varepsilon \int_\omega (A_{12}.\nabla W_1)W_1\Big)\,,
\end{aligned}
\end{equation}
  and
  \begin{equation}
\label{eq:52}
    \int_{\Omega_\infty^+} v_\varepsilon^2=\int_0^\infty e^{-2\varepsilon
  x_1}\,\big(=\frac{1}{2\varepsilon}\big)\,.
  \end{equation}
 By \eqref{eq:51}--\eqref{eq:52} we obtain
 \begin{equation*}
   \frac{ \int_{\Omega_\infty^+} A\nabla v_\varepsilon.\nabla
   v_\varepsilon}{\int_{\Omega_\infty^+} v_\varepsilon^2}=\mu^1-2\varepsilon \int_\omega (A_{12}.\nabla W_1)W_1+\varepsilon^2 \int_\omega
a_{11}W_1^2\,,
 \end{equation*}
 so by sending $\varepsilon$ to $0$ we deduce that $\nu_\infty^+\leq \mu^1$.
\end{proof}
It is easy to identify $\nu_\infty^\pm$ with the limits, as
$\ell\to\infty$, of certain minimization problems on
$\Omega_\ell^{\pm}$. This is the content of the next lemma (see
\eqref{dec} and \eqref{eq:40}  for the definitions of
$\gamma_\ell^\pm$ and $\Gamma_\ell^\pm$).
\begin{lemma}
  \label{lem:lim-semi}
   We have $\nu_\infty^\pm=\lim_{\ell\to\infty}
   \tilde{\lambda}_\ell^{1,\pm}\,,$ where
\begin{equation}
   \label{eq:35}
  \tilde{\lambda}_\ell^{1,\pm} = \inf\{ \int_{\Omega_\ell^{\pm}}A\nabla u.\nabla u \,: \,
  u\in H^1(\Omega_\ell^{\pm}),\, \int_{\Omega_\ell^{\pm}}u^2 = 1, u=0 \text{ on
  }\gamma_\ell^{\pm}\cup\Gamma_\ell^\pm\}\,.
 \end{equation}
\end{lemma}
\begin{remark}
\label{rem:uellt}
  It is a standard fact that the infimum in \eqref{eq:35} is actually
  attained. The unique positive normalized minimizers will be denoted
  by $\tilde u_\ell^\pm$.
\end{remark}
\begin{proof}
   We present the proof for $\tilde{\lambda}_\ell^{1,+}$ as the
   proof for $\tilde{\lambda}_\ell^{1,-}$is completely
 analogous. Note first that the limit $\lim_{\ell\to\infty}
   \tilde{\lambda}_\ell^{1,+}$ exists since the function $\ell\mapsto
   \tilde{\lambda}_\ell^{1,+}$ is non increasing. Indeed, if
   $\ell_1<\ell_2$ then any admissible function in \eqref{eq:35} for
   $\tilde{\lambda}_{\ell_1}^{1,+}$ can be extended to an admissible
   function for $\tilde{\lambda}_{\ell_2}^{1,+}$ by setting it to zero
   on $\Omega_{\ell_2}^+\setminus \Omega_{\ell_1}^+$. A similar
   argument shows that 
 $\tilde{\lambda}_\ell^{1,+}\geq \nu_\infty^+$, for any $\ell>0$.
On the other hand, the density of the space 
\begin{equation}
\label{eq:41}
V_s(\Omega_{\infty}^+) =  \{ u\in C^{\infty}(\Omega_{\infty}^+)\cap V(\Omega_\infty^+)\,:\,
\exists  M=M(u) > 0 \text{ s.t. } u = 0 \text{ on }  (M, \infty)\times\omega \}\,,
\end{equation}
in $V(\Omega_\infty^+)$ implies that for each $u\in
V(\Omega_\infty^+)\setminus\{0\}$ and any $\varepsilon>0$ we can find
an $\ell_\varepsilon$ and $v_\varepsilon\in V_s(\Omega_{\infty}^+)$
with $\text{supp}(v_\varepsilon)\subset \Omega_{\ell_\varepsilon}^{+}$
such that
\begin{equation*}
 \left|\frac{\int_{\Omega_\infty^+} (A\nabla v_\varepsilon) .\nabla v_\varepsilon
  }{\int_{\Omega_\infty^+} v_\varepsilon^2}-\frac{\int_{\Omega_\infty^+} (A\nabla u).\nabla
  u}{\int_{\Omega_\infty^+} u^2}\right|\leq \varepsilon\,,
\end{equation*}
 and \eqref{eq:35}  follows (for $\tilde{\lambda}_\ell^{1,+}$).
\end{proof}
 Our next result  complements the result of \rth{th1} in two ways: by showing that
 the limit $\lim_{\ell \to \infty }  \lambda_\ell^1$ exists and by
 identifying its value. 
 \begin{theorem}
   \label{th:lim-lam}
  We have 
\begin{equation}\label{eq:lim-lam} 
 \lim_{\ell \to \infty }  \lambda_\ell^1 =\min(\nu_\infty^+,\nu_\infty^-)\,.
\end{equation}
 \end{theorem}
 \begin{proof}
  {\rm (i)} We shall first show that
  \begin{equation}
    \label{eq:42}
    \limsup_{\ell\to\infty} \lambda_\ell^1\leq \min(\nu_\infty^+,\nu_\infty^-)\,.
  \end{equation}
 We may assume w.l.o.g.~that $\nu_\infty^+=
 \min(\nu_\infty^+,\nu_\infty^-)$. Given $\varepsilon>0$ we  may find
 by \rlemma{lem:lim-semi} an $\ell_\varepsilon>1/\varepsilon$ such that
 $\tilde{\lambda}_{\ell_\varepsilon}^{1,+}\leq \nu_\infty^++\varepsilon$. Since $\lambda_{\ell/2}^1\leq \tilde{\lambda}_\ell^{1,+}$ by
 the definitions \eqref{reiligh} and \eqref{eq:35}, we
 easily deduce \eqref{eq:42}.\\
  {\rm (ii)} We now treat the case where \eqref{con} holds. Let $u_\ell$ denote the positive normalized minimizer in
  \eqref{reiligh}. Define $v_\ell(x)=\rho(x_1)u_\ell(x)$ where $\rho$
  is given by
\begin{equation}
\label{eq:43}
\rho(x_1)=
\begin{cases}
   0 & x_1\leq -1\,,\\
1+x_1& x_1\in (-1,0)\,,\\
   1 & x_1\geq 0\,.
   \end{cases}
\end{equation}
 By \eqref{mod} and \eqref{eq:43} we have
 \begin{equation}
   \label{eq:44}
   \int_{(-1,\ell)\times\omega} (A\nabla v_\ell).\nabla v_\ell\leq
   \int_{\Omega_\ell^+}(A\nabla u_\ell).\nabla u_\ell+C_A \int_{(-1,0)\times\omega}|\nabla v_\ell|^2\,.
 \end{equation}
 Define $w_{\ell+1}(x_1,X_2)=v_\ell(x_1+\ell,X_2)$ on
 $\Omega_{\ell+1}^-$ and notice that it is an admissible function for
 the infimum defining $\tilde{\lambda}_{\ell+1}^{1,-}$ (see
 \eqref{eq:35}). By \eqref{eq:44} and \eqref{eq:27}--\eqref{eq:32} we
 obtain, for some positive constant $C$,
 \begin{equation}
   \label{eq:45}
   \int_{\Omega_{\ell+1}^-} (A\nabla w_{\ell+1}).\nabla w_{\ell+1}\leq \int_{\Omega_\ell^+}(A\nabla u_\ell).\nabla u_\ell+C\alpha^\ell\,.
 \end{equation}
  Denote
  \begin{equation*}
    N_\ell^\pm=\int_{\Omega_\ell^\pm}(A\nabla u_\ell).\nabla
    u_\ell~\text{ and }~ D_\ell^\pm=\int_{\Omega_\ell^\pm} |u_\ell|^2\,,
  \end{equation*}
 so that in particular we have 
 \begin{equation}
\label{eq:47}
   N_\ell^++N_\ell^-=\lambda_\ell^1~\text{ and }~D_\ell^++D_\ell^-=1\,.
 \end{equation}
 By \eqref{eq:45} and an analogous construction on $\Omega_{\ell+1}^+$
 we have
 \begin{equation}
\label{eq:46}
   \tilde{\lambda}_{\ell+1}^{1,-}\leq
   \frac{N_\ell^++C\alpha^\ell}{D_\ell^+}~\text{ and }~\tilde{\lambda}_{\ell+1}^{1,+}\leq
   \frac{N_\ell^-+C\alpha^\ell}{D_\ell^-}\,.
 \end{equation}
 From \eqref{eq:46} and \eqref{eq:47} it follows that
 \begin{equation}
\label{eq:48}
   \min \{
   \tilde{\lambda}_{\ell+1}^{1,-},\tilde{\lambda}_{\ell+1}^{1,+}\}\leq
   D_\ell^+\tilde{\lambda}_{\ell+1}^{1,-}+D_\ell^-\tilde{\lambda}_{\ell+1}^{1,+}\leq \lambda_\ell^1+C\alpha^\ell\,.
 \end{equation}
 Passing to the limit $\ell\to\infty$ in \eqref{eq:48} and using \rlemma{lem:lim-semi} yields
 \begin{equation}
   \label{eq:49}
   \min(\nu_\infty^+,\nu_\infty^-)\leq \liminf_{\ell\to\infty} \lambda_\ell^1\,,
 \end{equation}
  which combined with \eqref{eq:42} clearly implies \eqref{eq:lim-lam}
  (when  \eqref{con} holds).\\
{\rm (iii)} Finally, we turn to the case where \eqref{con} doesn't
hold. In this case  we know already
from \rth{th1} that $\lambda_\ell^1=\mu^1$ for all $\ell$. The proof
of \eqref{eq:lim-lam} will be clearly completed if we show that 
 $ \nu_\infty^+=\nu_\infty^-=\mu^1$.
 We shall only show that $\nu_\infty^+=\mu^1$ as the argument for $\nu_\infty^-$ is
 identical. By \rlemma{lem:nui} we have $\nu_\infty^+\leq\mu^1$. 
 The reverse inequality is a special case of \rth{th:infinity}\,(ii), see below.
 \end{proof}
 The argument of the above proof can be used to derive an additional information that
 will be useful in the next section.
 \begin{proposition}
   \label{prop:side}
 If $\nu_\infty^+<\nu_\infty^-$ then $\lim_{\ell\to\infty}\int_{\Omega_\ell^+}|\nabla u_\ell|^2+|u_\ell|^2=0$.
 \end{proposition}
 \begin{proof}
  We use the same notation as in the proof of \rth{th:lim-lam}. Passing to the limit $\ell\to\infty$ in \eqref{eq:48}, using
   \rlemma{lem:lim-semi} and \eqref{eq:lim-lam} yields
   \begin{equation*}
     \Big(\limsup_{\ell\to\infty}
     D_\ell^+\Big)\nu_\infty^-+\Big(1-\limsup_{\ell\to\infty}
     D_\ell^+)\nu_\infty^+\leq \lim_{\ell\to\infty} \lambda_\ell^1=\nu_\infty^+\,,
   \end{equation*}
 so necessarily $ \limsup_{\ell\to\infty} D_\ell^+=0$.
 Next, by \eqref{eq:46} we have for $\ell$ large,
 \begin{equation}
  \label{eq:70}
 \frac{N_\ell^+}{D_\ell^-}+\tilde{\lambda}_{\ell+1}^{1,+}-C\alpha^\ell\leq
    \frac{N_\ell^++N_\ell^-}{D_\ell^-}\leq \frac{N_\ell^++N_\ell^-}{D_\ell^++D_\ell^-}=\lambda_\ell^1\,.
 \end{equation}
 Since in our case, $\lim_{\ell\to\infty}
 \lambda_\ell^1=\lim_{\ell\to\infty}
 \tilde{\lambda}_{\ell+1}^{1,+}=\nu_\infty^+$, and we know already
 that $\lim_{\ell\to\infty}D_\ell^-=1$, 
 we deduce from \eqref{eq:70} that $\lim_{\ell\to\infty}N_\ell^+=0$.
 \end{proof}

\section{The problem on a semi-infinite cylinder}
\label{sec:cylinder} 
 In this section we further investigate the minimization problem
 \eqref{eq:33}. By Remark~\ref{rem:reflection} it is enough to
 consider $\nu_\infty^+$. There are two main questions we are interested
 in. First, we want to identify the conditions under which the infimum
 in \eqref{eq:33} is attained. Second, we would like to know when the
 inequality $\nu_\infty^+<\mu^1$ hold. The next proposition  shows that
 the two
 questions are closely related to each other.
 \begin{proposition}
   \label{prop:existence}
 If 
 \begin{equation}
   \label{eq:34}
   \nu_\infty^+<\mu^1 \,,
 \end{equation}
then $\nu_\infty^+$ is attained. The minimizer $\tilde u^+$ is
unique up to multiplication by a constant, has constant sign and satisfies 
 \begin{equation}
 \label{eq:37}
 \left\{
 \begin{aligned}
   -&\Div(A(X_2)\nabla \tilde u^+)=\nu_\infty^+\tilde u^+  \quad\text{ in }\Omega_\infty^+\,,\\
      &\tilde u^+=0 \text{ on } \gamma_\infty^{+}\,,\\
    & (A(X_2)\nabla \tilde u^+).\nu=0 \text{ on } \{0\}\times\omega\,.
 \end{aligned}
 \right.
 \end{equation}
 \end{proposition}
\begin{proof}
The existence of a minimizer will be achieved by taking the limit
$\ell\to\infty$ of the minimizers $\{\tilde u_\ell^+\}$ in \eqref{eq:35}
(see \rrem{rem:uellt}). Since $\{\tilde u_\ell^+\}$ is bounded
in $H^1(\Omega_\infty^+)$, a subsequence $\{\tilde u_{\ell_k}^+\}$ converges
 weakly to some limit $\tilde u^+\in H^1(\Omega_\infty^+)$. Take any $\varphi\in
V_s(\Omega_\infty^+)$. Since
 $\nu_\infty^+=\lim_{k\to\infty} {\tilde\lambda}_{\ell_k}^{1,+}$ by
 \rlemma{lem:lim-semi}, we can pass to the limit in the following
 equality, that holds for $\ell_k>M(\varphi)$ (see \eqref{eq:41}),
 \begin{equation*}
   \int_{\Omega_\infty^+} A\nabla \tilde u_{\ell_k}^+\cdot\nabla\varphi= 
 {\tilde\lambda}_{\ell_k}^{1,+}\int_{\Omega_\infty^+} \tilde u_{\ell_k}^+\varphi\,,
 \end{equation*}
and obtain that
\begin{equation}
\label{eq:58}
   \int_{\Omega_\infty^+} A\nabla \tilde u^+\cdot \nabla\varphi= 
 \nu_\infty^+ \int_{\Omega_\infty^+} \tilde u^+\varphi\,.
 \end{equation}
 Since \eqref{eq:58} is valid for any $\varphi\in
 V_s(\Omega_\infty^+)$, and by density also for any $\varphi\in
 V(\Omega_\infty^+)$, we obtain that $\tilde u^+$ is a solution of
 \eqref{eq:37}. To conclude that it is a minimizer realizing
 $\nu_\infty^+$ in \eqref{eq:33} we only need to prove that it is {\em
   nontrivial}, i.e., that $\tilde u^+\not\equiv 0$. Actually, we are
 going to show that $\int_ {\Omega_\infty^+} (\tilde u^+)^2=1$ and $\tilde
 u^+>0$. For that matter we will prove decay estimates for $\tilde
 u_\ell^+$ for large $x_1$, that imply concentration near $x_1=0$, using
 the same technique as the one used in the proof of \rth{th:decay}.

Let $\ell$ and $\ell'$ satisfy  $0< \ell^{'} \leq \ell-1$. Define $\tilde\rho_{\ell^{'}}=\tilde\rho_{\ell^{'}}(x_1)$ by
\begin{equation*}
\tilde\rho_{\ell^{'}}(x_1)=
\begin{cases}
   0 & x_1\leq \ell'\,,\\
x_1-\ell'& x_1\in (\ell',\ell'+1)\,,\\
   1 & x_1\geq \ell'+1\,.
   \end{cases}
\end{equation*}
 By the Euler-Lagrange equation satisfied by $\tilde u_\ell^+$ we have
\begin{equation*}
\int_{\Omega^+_\ell} (A\nabla \tilde u_\ell^+). \nabla
(\tilde\rho_{\ell^{'}}^2\tilde u_\ell^+) = \tilde\lambda_\ell^{1,+}
\int_{\Omega^+_\ell}\tilde\rho_{\ell^{'}}^2|\tilde u_\ell^+|^2\,.
\end{equation*}
Repeating the argument used to derive \eqref{zero1} we obtain
\begin{equation}
\label{eq:59}
\begin{aligned}
(\sigma_{\ell/2}^1 -\tilde\lambda_\ell^{1,+}) \int_{\Omega_\ell^+\setminus \Omega_{\ell'+1}}\!\!\!|\tilde u_\ell^+|^2 &\leq (\sigma_{\ell/2}^1 -\tilde\lambda_\ell^{1,+}) \int_{\Omega_\ell^+} |\tilde u_\ell^+|^2 \tilde\rho_{\ell^{'}}^2 \leq
\int_{\Omega_\ell^+}|\tilde u_\ell^+|^2 (A\nabla
\tilde\rho_{\ell^{'}}).\nabla \tilde\rho_{\ell^{'}}\\&=
\int_{\Omega_{\ell'+1}^+\setminus \Omega_{\ell'}}\!\!|\tilde u_\ell^+|^2 (A\nabla
\tilde\rho_{\ell^{'}}).\nabla \tilde\rho_{\ell^{'}}
\leq C_A  \int_{\Omega_{\ell'+1}^+\setminus \Omega_{\ell'}}\!\!\!|\tilde u_\ell^+|^2\,.
\end{aligned}
\end{equation}
Using \eqref{upplow}  together with \eqref{eq:34} and
\rlemma{lem:lim-semi} we deduce that there exist
$\tilde\ell_0>0$ and $\tilde \beta>0$ such that for
$\ell>\tilde\ell_0$ we have $\sigma_{\ell/2}^1
-\tilde\lambda_\ell^{1,+}\geq\tilde\beta$. Therefore, we deduce from
\eqref{eq:59} that
\begin{equation}
  \label{eq:60}
  \int_{\Omega_\ell^+\setminus \Omega_{\ell'+1}}
  |\tilde u_\ell^+|^2\leq \tilde\alpha \int_{\Omega_\ell^+\setminus
    \Omega_{\ell'}}\!\!\!|\tilde u_\ell^+|^2~\text{ with }~\tilde\alpha:=\frac{C_A}{\tilde\beta+C_A}\,.
\end{equation}
 Fix any $r>1$. Applying \eqref{eq:60} successively for $\ell^{'} = r-1, r-2,\ldots,r-[r]$ yields
\begin{equation*}
\int_{\Omega_\ell^+\setminus\Omega_r} |\tilde u_\ell^+|^2 \leq
\tilde\alpha^{[r]}\int_{\Omega^+_\ell}|\tilde u_\ell^+|^2  =\tilde\alpha^{[r]},~\forall\ell>r\,.
\end{equation*}
In other words,
\begin{equation}
  \label{eq:61}
  \int_{\Omega_r^+} |\tilde u_\ell^+|^2\geq 1-\tilde\alpha^{[r]}\,.
\end{equation}
 Since $\tilde u_{\ell_k}\to\tilde u^+$ strongly in
 $L^2(\Omega_r^+)$, we deduce from \eqref{eq:61} that 
 \begin{equation}
   \label{eq:62}
  \int_{\Omega_r^+} (\tilde u^+)^2\geq 1-\tilde\alpha^{[r]}\,.
\end{equation}
 This already implies that $\tilde u^+$ is a nontrivial nonnegative solution to
 \eqref{eq:37} and therefore, a minimizer in \eqref{eq:33}. Applying
 \eqref{eq:62} with arbitrary large $r$, we get that
 $\int_{\Omega_\infty^+} (\tilde u^+)^2=1$. The uniqueness of the minimizer
 follows by a standard argument, using the fact that any minimizer
 must have a constant sign.
\end{proof}
\noindent\underline{Open Problem:}
Is it true that \eqref{eq:34} is  also a necessary condition for
 the existence of a minimizer realizing $\nu_\infty^+$? In \rth{th:infinity} below we will show
 nonexistence of a minimizer when $\nu_\infty^+=\mu^1$, but under the
 additional condition \eqref{eq:53}.\\[2mm]
 The next result provides a
 sufficient condition for \eqref{eq:34} to hold and another one for it
 to fail. 
\begin{theorem}
  \label{th:infinity}
{\rm (i)} Assume that \eqref{con} is satisfied. If the following condition holds,
\begin{equation}
  \label{eq:39}
  \int_\omega (A_{12}.\nabla W_1)W_1\geq 0\,,
\end{equation}
 then \eqref{eq:34} holds. \\
{\rm (ii)} If 
 \begin{equation}
   \label{eq:53}
 A_{12}.\nabla W_1\leq 0\text{ a.e. in }\omega
 \end{equation}
 then $\nu_\infty^+=\mu^1$. Moreover, in this case there is no
 minimizer realizing $\nu_\infty^+$.
\end{theorem}
\begin{proof}
(i) Assume that \eqref{eq:39} is satisfied.  A similar computation to the one done in the proof of
 \rth{th:lim-zero} (see also Remark~\ref{rem:Ww})  shows that $\{\tilde
 v_\ell^\varepsilon\}$ given by \eqref{eq:22}, satisfy not only
 \eqref{eq:23}, but also 
\begin{equation*}
\inf_{\varepsilon>0} \lim_{\ell\to0}\frac{\int_{\Omega_\ell^-} (A\nabla \tilde v_\ell^\varepsilon).\nabla\tilde
  v_\ell^\varepsilon}{\int_{\Omega_\ell^-}|\tilde v_\ell^\varepsilon|^2}=\int_{\omega}(A_{22}\nabla  W_1).\nabla 
  W_1-\frac{|A_{12}\nabla  W_1|^2}{a_{11}}\,.
 \end{equation*}
 Indeed, we only need to note that the term corresponding to the
 second term
 on the RHS of \eqref{eq:69} is of the order $O(\ell^2)$. Hence, we can fix values of
$\ell_1$ and $\varepsilon_1$ such that the  following analog of
\eqref{eq:19} holds,
\begin{equation}
  \label{eq:54}
-\gamma_1:=\int_{\Omega_{\ell_1}^-} (A\nabla\tilde
v_{\ell_1}^{\varepsilon_1}).\nabla\tilde
v_{\ell_1}^{\varepsilon_1}-\mu^1\int_{\Omega_{\ell_1}^-}|\tilde v_{\ell_1}^{\varepsilon_1}|^2<0\,.
\end{equation}
For each $\alpha>0$ we define a  test function in
$V_\infty(\Omega_\infty^+)$ by
  \begin{equation*}
    z_\alpha(x_1,X_2)=
    \begin{cases}
       \tilde v_{\ell_1}^{\varepsilon_1}(x_1-\ell_1,X_2) & x_1\in [0,\ell_1)\,,\\
       W_1(X_2)e^{-\alpha (x_1-\ell_1)} & x_1 \in[\ell_1,\infty)\,.
    \end{cases}
  \end{equation*}
Above we used the fact that  $\tilde v_{\ell_1}^{\varepsilon_1}(0,X_2)=W_1(X_2)$.  We have,
\begin{equation*}
  \int_{\Omega_\infty^+}|z_\alpha|^2=\int_{\Omega_{\ell_1}^-} |\tilde v_{\ell_1}^{\varepsilon_1}|^2+ (\int_0^\infty\!\!
  e^{-2\alpha x_1})\int_\omega W_1^2=\int_{\Omega_{\ell_1}^-}
  |\tilde v_{\ell_1}^{\varepsilon_1}|^2+\frac{1}{2\alpha}\,,
\end{equation*}
and 
 \begin{equation*}
\begin{aligned}
   \int_{\Omega_\infty^+} (A\nabla z_\alpha).\nabla z_\alpha&=
\int_{\Omega_{\ell_1}^-} (A\nabla \tilde v_{\ell_1}^{\varepsilon_1}).\nabla
\tilde v_{\ell_1}^{\varepsilon_1}\\
&\phantom{=}+ \frac{1}{2\alpha}\left(\alpha^2\int_\omega
  a_{11}W_1^2-2\alpha \int_\omega (A_{12}.\nabla
  W_1)W_1+\int_\omega (A_{22} \nabla W_1).\nabla W_1\right)
\end{aligned}
 \end{equation*}
  Therefore, using \eqref{eq:54} we get
\begin{equation}
\label{eq:57}
\nu_\infty^+-\mu^1\leq \frac{ \int_{\Omega_\infty^+} A\nabla z_\alpha.\nabla
  z_\alpha}{\int_{\Omega_\infty^+}|z_\alpha|^2}- \mu^1
<\frac{ \frac{\alpha}{2}\int_\omega
  a_{11}W_1^2-\int_\omega (A_{12}.\nabla
  W_1)W_1 -\gamma_1}          {\int_{\Omega_{\ell_1}^-}
  |v_{\ell_1}^{\varepsilon_1}|^2+\frac{1}{2\alpha}}\,.
\end{equation}
 Since $\gamma_1>0$ and $\int_\omega (A_{12}.\nabla
  W_1)W_1\geq 0$ by \eqref{eq:39}, it is clear that we can choose
  $\alpha$ small enough to ensure that the RHS of \eqref{eq:57} is
  negative, completing the proof of \eqref{eq:34}.\\
(ii) We notice that not only $V_s(\Omega_+^\infty)$ is dense in
$V(\Omega_+^\infty)$ (see \eqref{eq:41}), but its subspace
\begin{equation*}
V_s^0(\Omega_{\infty}^+) =  \Big\{ u\in V_s(\Omega_\infty^+)\,:\,
\exists  \delta=\delta(u) > 0 \text{ s.t. } u(x) = 0 \text{ for }
\text{dist}(x,\gamma_\infty^+)\leq \delta \Big\}\,,
\end{equation*}
is dense as well. By elliptic regularity and the strong maximum principle we
know that $W_1$ is continuous and positive in $\omega$ (see
\cite[Chapter 8]{gt}). We shall use
the following version of Picone identity,
\begin{equation}
  \label{eq:55}
  (A\nabla u).\nabla u-(A\nabla 
  v).\nabla\big(\frac{u^2}{v}\big)=A\big(\nabla u-\frac{u}{v}\nabla
  v\big).\big(\nabla u-\frac{u}{v}\nabla v\big)\geq 0\,.
\end{equation}
 Using \eqref{eq:55} with any $u\in V_s^0(\Omega_{\infty}^+)$ and
 $v=W_1$, integrating and applying the generalized Green formula yields
 \begin{equation}
   \label{eq:64}
\begin{aligned}
   0&\leq\int_{\Omega_{\infty}^+}  A\big(\nabla u-\frac{u}{W_1}\nabla
  W_1\big).\big(\nabla u-\frac{u}{W_1}\nabla W_1\big)=
\int_{\Omega_{\infty}^+} (A\nabla u).\nabla u-(A\nabla
   W_1).\nabla\big(\frac{u^2}{W_1}\big)\\&=
 \int_{\Omega_{\infty}^+} (A\nabla u).\nabla u
+\int_{\Omega_{\infty}^+}
 \Div(A\nabla W_1)\big(\frac{u^2}{W_1}\big)-\int_{\{0\}\times\omega}
 (A\nabla W_1.\nu)\big(\frac{u^2}{W_1}\big)\\
&=\int_{\Omega_{\infty}^+} (A\nabla u).\nabla
u-\mu^1u^2+\int_{\omega} \Big(A_{12}.\nabla W_1\Big)\frac{u^2(0,X_2)}{W_1(X_2)}\,.
\end{aligned} 
\end{equation}
By \eqref{eq:64} and \eqref{eq:53} we deduce that 
\begin{equation}
  \label{eq:65}
  0\leq \int_{\Omega_{\infty}^+}  A\big(\nabla u-\frac{u}{W_1}\nabla
  W_1\big).\big(\nabla u-\frac{u}{W_1}\nabla W_1\big)\leq \int_{\Omega_{\infty}^+} A\nabla u.\nabla u-\mu^1u^2\,.
\end{equation}
 By the density of $V_s^0(\Omega_{\infty}^+)$ in
 $V(\Omega_{\infty}^+)$ it follows that \eqref{eq:65} holds for every
 $u\in  V(\Omega_{\infty}^+)$, i.e., $\nu_\infty^+\geq
 \mu^1$. Finally, applying \eqref{eq:63} we get that $\nu_\infty^+=
 \mu^1$. To conclude, assume by negation that $\nu_\infty^+$ is
 realized by a minimizer $u$. Then, by \eqref{eq:65} we get that
 $\nabla\big(\frac{u}{W_1}\big)=0$ a.e., implying that $u=cW_1$ for
 some constant $c\neq 0$. But this is clearly a contradiction since
 $W_1\not\in V(\Omega_\infty^+)$.
\end{proof}
\begin{remark}
   An immediate consequence of \rth{th:infinity} and
   \rrem{rem:reflection} is that if \eqref{con} holds and
   $\int_\omega (A_{12}.\nabla W_1)W_1= 0$, then we have both
   $\nu_\infty^+<\mu^1$ and $\nu_\infty^-<\mu^1$. A special case is
   when property (S) holds. Another direct consequence is that
   whenever \eqref{con} holds we have
   $\min(\nu_\infty^+,\nu_\infty^-)<\mu^1$. However, this fact follows
   already from our previous results,  by combining \rth{th1} and \rth{th:lim-lam}.
\end{remark}
 Our last result provides a description of the asymptotic profile of the
 eigenfunctions $\{u_\ell\}$ near the ends of the cylinder. We denote
 by $\tilde u^\pm$ the unique positive renormalized minimizer for
 $\nu_\infty^\pm$, when it exists. For each
 $\ell>0$ we define:
 \begin{equation}
   \label{eq:66}
\begin{aligned}
   \tilde v_\ell^+(x_1,X_2)=u_\ell(x_1-\ell,X_2)~\text{ on
   }~\Omega_\ell^+\,,\\
 \tilde v_\ell^-(x_1,X_2)=u_\ell(x_1+\ell,X_2)~\text{ on
   }~\Omega_\ell^-\,.
\end{aligned}
 \end{equation}
 The next theorem describes two possible scenarios that may occur:
 concentration near one of the ends of the cylinder, or concentration
 near both ends.
 \begin{theorem}
\label{th:descr}
   {\rm(i)} If $\nu_\infty^+<\nu_\infty^-$ then, for every $r>0$,
   \begin{equation}
     \label{eq:67}
     \tilde v_\ell^+\to \tilde u^+ \text{ in }H^1(\Omega_r^+) ~\text{
       and }~ \tilde v_\ell^-\to 0 \text{ in }H^1(\Omega_r^-)\,.
   \end{equation}
{\rm(ii)} If both \eqref{eq:37} and property  (S) hold then we have $\tilde
v_\ell^+(x_1,X_2)=\tilde v_\ell^-(-x_1,-X_2)$ and for every $r>0$,
\begin{equation}
  \label{eq:68}
  \tilde v_\ell^+\to \tilde u^+ \text{ in }H^1(\Omega_r^+) ~\text{
       and }~ \tilde v_\ell^-\to \tilde u^-\text{ in }H^1(\Omega_r^-)\,.
\end{equation}
 \end{theorem}
 \begin{proof}
  (i)   The convergence of $\{ \tilde v_\ell^-\}$ to $0$ in
  $H^1(\Omega_r^-)$ for all $r>0$ is clear from \rprop{prop:side}, so
  we only need to prove the result for $\{ \tilde v_\ell^+\}$. Since
  $\{ \tilde v_\ell^+\}$ is bounded in $H^1(\Omega_\ell^+)$, given any sequence $\ell_k\to\infty$, we can apply a
  diagonal argument to  $\{ \tilde v_{\ell_k}^+\}$ to extract a subsequence, still denoted by $\{\ell_k\}$, 
  such that $\tilde v_{\ell_k}^+$ converges weakly
  in $H^1(\Omega_r^+)$ and strongly in $L^2(\Omega_r^+)$ to some
  function $v^+\in H^1(\Omega_\infty^+)$, for
  every $r>0$. By \eqref{eq:27} and \rprop{prop:side} we have
  \begin{equation}
    \label{eq:56}
    \int_{\Omega_r^+} |\tilde v_\ell^+|^2=\int_{\Omega_\ell^-\setminus
      \Omega_{\ell-r}}\!\!|u_\ell|^2=1-\int_{\Omega_{l-r}^-}\!
      |u_\ell|^2-\int_{\Omega_\ell^+} \!|u_\ell|^2
               \geq 1-\alpha^{[r]}+o(1)\,,
  \end{equation}
 where $o(1)$ stands for a quantity that tends to $0$ when $\ell\to\infty$.
 Passing to the limit in \eqref{eq:56} with $\ell=\ell_k$, yields,
 \begin{equation}
   \label{eq:71}
   \int_{\Omega_r^+} |v^+|^2\geq 1-\alpha^{[r]}\,,
 \end{equation}
 and since $r$ is arbitrary, we get that
 $\int_{\Omega_\infty^+}|v^+|^2=1$.  In addition, we clearly have 
 \begin{multline}
\label{eq:72}
   \nu_\infty^+=\lim_{\ell \to\infty} \lambda_\ell^1\geq \lim_{\ell\to\infty} \int_{\Omega_\ell}
   (A\nabla u_\ell).\nabla u_\ell \\\geq
   \limsup_{k\to\infty}  \int_{\Omega_r^+} (A\nabla\tilde
   v_{\ell_k}^+).\nabla\tilde v_{\ell_k}^+\geq
\int_{\Omega_r^+} (A\nabla v^+).\nabla v^+\,.
 \end{multline}
From \eqref{eq:71}--\eqref{eq:72} we deduce that $\int_{\Omega_\infty^+}
(A\nabla v^+).\nabla v^+=\nu_\infty^+$, i.e., $v^+$ is a nonnegative normalized
minimizer, realizing $\nu_\infty^+$ in \eqref{eq:33}. Therefore, it
must coincide with $\tilde u^+$. Finally, defining on
$(0,\infty)$ the function
\begin{equation*}
  f(r)=\limsup_{k\to\infty}  \int_{\Omega_r^+} (A\nabla\tilde
   v_{\ell_k}^+).\nabla\tilde v_{\ell_k}^+-
\int_{\Omega_r^+} (A\nabla\tilde u^+).\nabla\tilde u^+ \,,
\end{equation*}
 we see that on the one hand it is a nonnegative and nondecreasing function,
 while on the other hand $\lim_{r\to\infty} f(r)=0$. Hence $f(r)\equiv 0$, implying the
 strong convergence $\tilde v_{\ell_k}^{+}\to \tilde u^+$ in
 $H^1(\Omega_r^+)$ for all $r>0$.
 The uniqueness of the possible limit implies the the same convergence
 holds for the whole family $\{\tilde v_{\ell}^{+}\}$.\\
{\rm (ii)} In this case we have the symmetry relation  $u_\ell(x_1,
X_2) = u_\ell(-x_1, -X_2)$ by \rprop{symmetry}, and the same argument as in {\rm (i)}
gives the result. 
\end {proof}
\begin{remark}
\label{rem:pincho}
   \rth{th:descr} provides a description of the profile of $u_\ell$ near
   the ends of the cylinder. As pointed to us by Y.~Pinchover, a
   description of the profile of $u_\ell$ in the bulk can be given using
   the characterization of positive solutions in an infinite cylinder,
   given in \cite{pincho}. Indeed, setting $v_\ell(x)=u_\ell(x)/u_\ell(0)$,
   and employing Harnack's inequality and the boundary Harnack
   principle (see \cite[Theorems 1.2 and 1.3]{murata})  we obtain a
   subsequence $\{v_{\ell_k}\}$ that converges uniformly on each
   cylinder $\Omega_r$, $r>0$, to a limit $v$. The function $v$  is a positive solution
   on the infinite cylinder $\Omega_\infty=(-\infty,\infty)\times \omega$ of
   $-\Div(A(X_2)\nabla v)=\lambda_\infty v$ satisfying $v=0$ on $\partial\Omega_\infty=(-\infty,\infty)\times
   \partial\omega$, where $\lambda_\infty=\lim_{\ell\to\infty} \lambda_\ell^1=\min(\nu_\infty^+,\nu_\infty^-)$ 
   (by \rth{th:lim-lam}). From \cite[Theorem~5.1]{pincho}
   (that handles a much more general situation) it follows  
   that such $v$ is a linear combination of one or two {\em exponential}
   solutions of the form $v_\alpha(x)=\Phi_\alpha(X_2)e^{\alpha x_1}$. In particular,
   when property $(S)$ holds it follows  that $v$ takes the form 
$$
v(x)=g(X_2)e^{\alpha x_1}+g(-X_2)e^{-\alpha x_1}\,
$$
for some $\alpha>0$, if \eqref{con} holds, and $v(x)=cW_1(X_2)$ if \eqref{con} doesn't hold.
\end{remark}
\section{Some additional results }
\label{sec:further}
So far we only studied the asymptotic behavior of the first eigenvalue
$\lambda_\ell^1$ and the corresponding eigenfunction $u_\ell$. The analogous
behavior of the other eigenvalues  $\lambda_\ell^2,\lambda_\ell^3,$ etc., is also of
interest. For the case of Dirichlet boundary condition this was done
in \cite{dir}. For our case of mixed boundary conditions we have the
following partial result for $ \lambda_\ell^2$.
\begin{theorem}\label{second2}
If property $(S)$ holds then 
\begin{equation*}
\lim_{\ell \to \infty} \lambda_\ell^2 = \lim_{\ell \to \infty} \lambda_\ell^1\,.
\end{equation*}
\end{theorem}
\begin{proof}
Define $h_\ell^-$ and $h_\ell^+$ on $\Omega_{\ell}$ by
\begin{equation*}
 h_\ell^-(x) = \begin{cases}
\tilde{u}_\ell^{+}(x_1+\ell,X_2) &\text{ on }  \Omega_{\ell}^-, \\
0 &\text{ on }  \Omega_{\ell}^+
\end{cases}
\end{equation*}
and 
\begin{equation*}
h_\ell^+= \begin{cases}
 \tilde{u}_\ell^{-}(x_1-\ell,X_2)  &\text{ on }  \Omega_{\ell}^+, \\
0 &\text{ on } \Omega_{\ell}^-, 
\end{cases}
\end{equation*}
where $\tilde{u}_\ell^{-}, \tilde{u}_\ell^{+}  $ are defined in
Remark~\ref{rem:uellt}.  Set  $\mathcal{H}_\ell = \alpha_\ell h_\ell^- + \beta_\ell h_\ell^+$,\ where
$\alpha_\ell, \beta_\ell$
are chosen such that 
\begin{equation*}
\int_{\Omega_{\ell}}u_{\ell}\mathcal{H}_\ell = 0~\text{ and }~\alpha_\ell^2+\beta_\ell^2>0\,.
\end{equation*}
Such a choice  is possible since we have to satisfy one linear
equation in two unknowns.
From the Rayleigh quotient characterization of $\lambda_\ell^2$ we get, since the functions $h_\ell^{+}$ and $h_\ell^{-}$ have disjoint supports,
\begin{multline}\label{sss}
\lambda_{\ell}^2 =\min
\left\{ \frac{\int_{\Omega_\ell}(A\nabla u).\nabla u}{\int_{\Omega_\ell} u^2}\,
  \big|\,0\neq u \in V(\Omega_\ell),\,\int_{\Omega_\ell}uu_\ell=0   \right\}
 \leq  \frac{\int_{\Omega_{\ell}}A\grad \mathcal{H}_\ell.\grad \mathcal{H}_\ell}{\int_{\Omega_{\ell}}\mathcal{H}_\ell^2}\\
= \frac{\alpha_\ell^2\int_{\Omega_{\ell}^{-}}(A\grad h_\ell^{-}).\grad h_\ell^{-} + \beta_\ell^2 \int_{\Omega_{\ell}^{+}}(A\grad h_\ell^{+}).\grad h_\ell^{+}}{\alpha_\ell^2\int_{\Omega_{\ell}^{-}}(h_\ell^{-})^2
 + \beta_\ell^2\int_{\Omega_{\ell}^{+}}(h_\ell^{+})^2}=\frac{\alpha_\ell^2\tilde{\lambda}_\ell^{1,+} + \beta_\ell^2\tilde{\lambda}_\ell^{1,-} }{\alpha_\ell^2 + \beta_\ell^2} .
\end{multline}
But the symmetry property  $(S)$ implies, by the same proof as that of
\rprop{symmetry}, that $\tilde u_\ell^+(x_1,X_2)=\tilde
u_\ell^-(-x_1,-X_2)$ and
$\tilde{\lambda}_\ell^{1,+}=\tilde{\lambda}_\ell^{1,-}$. Therefore, \eqref{sss}
implies that the RHS of \eqref{sss} equals $\tilde{\lambda}_\ell^{1,+}$ and we
obtain that 
\begin{equation*}
\lambda_{\ell}^1 < \lambda_{\ell}^2 \leq \tilde{\lambda}_\ell^{1,+}=\tilde{\lambda}_\ell^{1,-}\,.
\end{equation*}
The theorem then follows from Lemma \ref{lem:lim-semi} and \rth{th:lim-lam}.
\end{proof}
\smallskip

In the previous sections we considered the case of a cylinder which
goes to infinity in one direction.  We now consider the more general
case of a domain that tends to infinity  in several directions. 
In the rest of the paper we set
$$\Omega_\ell = (-\ell, \ell)^{p}\times \omega,$$
where $1\leq p < n$ and $\omega$ is a bounded subset of $\R^{n-p}$.  
The points in $\Omega_\ell$  are denoted by 
$$
X = (X_1,X_2) \text{ with }X_1 = (x_1,\ldots , x_p)\text{ and }X_2 = (x_{p+1},\ldots,x_n)\,.
$$
Let $A(X_2)$ be a $n \times n$ symmetric, positive definite matrix,
uniformly elliptic and uniformly bounded on $\omega$, as in the previous
sections. Now we consider the following decomposition to sub-matrices:
\begin{equation*}\label{matrix1}
\begin{aligned}
 A(X_2)=
\begin{pmatrix}
    A_{11}(X_2)  & A_{12}(X_2) \\
A_{12}^t(X_2) & A_{22}(X_2)
  \end{pmatrix}\end{aligned}
 \end{equation*}
where  $A_{11}, A_{12}$ and $A_{22}$ are $p\times p, p\times  (n-p) $ and
$(n-p)\times  (n-p)$ matrices, respectively.  We still denote by $\mu^1$ and
$W_1$ the first eigenvalue and the corresponding eigenfunction for the
problem \eqref{eq:3}. Let $C_i$ denote the $i$-th row of the 
matrix $A_{12}$, 
and denote by  $B_i$  the $(n-p+1)\times(n-p+1)$ matrix 
\begin{equation*}
\begin{aligned}
 B_i(X_2)=
\begin{pmatrix}
    a_{ii}(X_2)  & C_i(X_2) \\
C_i^t(X_2) & A_{22}(X_2)
 \end{pmatrix}\end{aligned}
\,,
\end{equation*}
for $1 \leq i \leq p$.
Since the matrix $B_i$ can be viewed as a representation of the restriction of the operator
 associated with $A$ to the subspace of $\R^n$ consisting of the
 vectors $v=(v_1,\ldots,v_n)$ satisfying $v_j=0$ for all $j$ such that
 $i\neq j\leq p$, we conclude
 that the matrices $B_i(X_2)$ are also uniformly elliptic for $X_2\in\omega$.

The following eigenvalue problem is the generalization of \eqref{eq:5}
to our setting:
\begin{equation}
\label{eq:74}
\left\{
\begin{aligned}
  -&\Div(A(X_2)\nabla u)=\sigma u \quad\text{ in }\Omega_\ell,\\
     &u=0 \quad\text{ on } (-\ell,\ell)^{p} \times \partial\omega,\\
     & (A(X_2)\nabla u).\nu=0 \quad\text{ on }\partial(-\ell,\ell)^p\times \omega.
\end{aligned}
\right.
\end{equation}
As before we denote by 
$\lambda_\ell^1$  the first eigenvalue and by $u_\ell$ the corresponding
normalized positive eigenfunction. We have the following
generalization of \rth{th1}.

\begin{theorem}
We have 
\begin{equation*}
\limsup_{\ell \to \infty }  \lambda_\ell^1 < \mu^1,
\end{equation*}
 provided the following condition holds,
\begin{equation}\label{cong}
A_{12}.\nabla W_1\not\equiv {\boldsymbol 0}\text{ a.e. on } \omega\,,
\end{equation}
 where  ${\boldsymbol 0}$ denotes the zero element in $\R^p$.
 In case \eqref{cong} does not hold we have $\lambda_\ell^1 = \mu^1$ for all $\ell>0$.
\end{theorem}

\begin{proof}
 Assume first that \eqref{cong} doesn't hold. Then there exists
 $i\in\{1,\ldots,p\}$ for which $(A_{12}\nabla W_1)_i$ is not identically zero
 (a.e.) on $\omega$. It follows that the hypotheses of \rth{th1} (for the case where
\eqref{con} holds) are satisfied for the eigenvalue problem associated
with the operator $-\Div(B_i(X_2)\nabla v)$ on the domain $\tilde
\Omega_\ell=(-\ell,\ell)\times\omega$ in $\R^{n-p+1}$.  Hence, there exist functions
$\phi_\ell(x_1,X_2)\in V(\tilde\Omega_\ell),\,\ell>0,$ such that
\begin{equation}
  \label{eq:75}
\limsup_{\ell \to \infty} \frac{\int_{\tilde \Omega_\ell} (B_i(X_2)
  \nabla\phi_\ell). \nabla\phi_{\ell}}{\int_{\tilde \Omega_\ell} \phi_\ell^2} < \mu^1\,.
\end{equation}
Define on $\Omega_\ell$, $v_\ell(X_1,X_2) := \phi_\ell(x_i,X_2)$. Noting that
$$\int_{\Omega_\ell} (A\grad v_\ell). \grad v_\ell = (2\ell)^{p-1} \int_{\tilde
  \Omega_\ell} (B_i \nabla\phi_\ell).\nabla\phi_{\ell} ~\text { and }~\int_{\Omega_{\ell}}
v_\ell^2=(2\ell)^{p-1} \int_{\tilde \Omega_\ell} \phi_\ell^2\,,$$
we get from \eqref{eq:75} that
$$\limsup_{\ell \to \infty} \lambda_\ell^1 \leq \limsup_{\ell \to \infty} \frac{\int_{\Omega_\ell} A\grad v_\ell\cdot \grad
  v_\ell}{\int_{\Omega_{\ell}} v_\ell^2} < \mu^1\,.$$
\noindent Assume now that \eqref{cong} does hold. Next we apply a
simple generalization of an argument from \rth{th:lim-zero}.
Let 
  $B=\begin{pmatrix}
B_{11} & B_{12} \\
B^t_{12} & B_{22}
\end{pmatrix}$ be a positive definite $n\times n$ matrix, where $B_{11}$ and
$B_{22}$ are $p\times p$ and $(n-p)\times(n-p)$ matrices, respectively. Represent any
vector $\vec{z}$ in $\R^n$ as $\vec{z}=(Z_1,Z_2)$ with $Z_1\in \R^{p}$
and $Z_2\in \R^{n-p}$. Then,
by a similar computation to the one leading to
\eqref{eq:obs}--\eqref{eq:13} we get  that for any fixed $Z_2\in\R^{n-p}$ we have
  \begin{equation}
\label{eq:76}
    \min_{Z_1\in\R^p} (B\vec{z}).\vec{z}=(B_{22}Z_2).Z_2-\big(B_{11}^{-1}B_{12}Z_2\big).B_{12}Z_2\,,
  \end{equation}
 and the minimum in \eqref{eq:76} is attained for
\begin{equation*}
Z_1=-B_{11}^{-1}(B_{12}Z_2)\,.
\end{equation*}
 Applying \eqref{eq:76} with $B=A(X_2)$ we obtain, for any $\ell>0$,
 \begin{equation}
\label{eq:77}
\begin{aligned}
   \int_{\Omega_\ell} (A(X_2)\nabla u_\ell).\nabla u_\ell&\geq  \int_{\Omega_\ell}
   (A_{22}\nabla_{X_2}u_\ell).\nabla_{X_2}u_\ell-\big(A_{11}^{-1}A_{12}\nabla_{X_2}u_\ell\big).A_{12}\nabla_{X_2}u_\ell\\
&\geq  \Lambda^1\int_{\Omega_\ell}u_\ell^2\,,
\end{aligned} 
 \end{equation}
 where $\Lambda^1$ is defined, generalizing \eqref{eq:Lambda}, by 
\begin{equation}
   \label{eq:78}
\Lambda^1=\inf\left\{ \int_{\omega}A_{22}\nabla u.\nabla
  u-\big(A_{11}^{-1}A_{12}\nabla u\big).A_{12}\nabla u
  :\,u \in H_0^1(\omega),\, \int_{\omega}u^2=1 \right\}.
 \end{equation}
The infimum in \eqref{eq:78} is attained by a unique positive function,
denoted again by $w_1$, that satisfies
\begin{equation}
\label{eq:79}
\left\{ 
\begin{aligned}
 -&\Div (A_{22}\nabla w_1)+\Div(A_{12}^tA_{11}^{-1}A_{12}\nabla w_1)=\Lambda^1w_1~\text{ in }\omega\,,\\
&w_1=0~\text{ on }\partial\omega\,.
\end{aligned}
\right.
\end{equation}
But if \eqref{cong} holds, then $W_1$ is also a positive eigenfunction
in \eqref{eq:79}, with eigenvalue $\mu^1$. As in the proof of \rth{th1}
we conclude that $\Lambda^1=\mu^1$ and the result follows from \eqref{eq:77}
(since clearly $\lambda_\ell^1\leq \mu^1$).
\end{proof}

\textbf{Acknowledgements:} \ 
The first author has been supported by the Swiss National Science
Foundation under the contract $\#$ 200021-129807/1.
 The third author thanks Yehuda Pinchover for Remark~\ref{rem:pincho}
 and for providing him the relevant references.

\end{document}